\documentclass[11pt]{amsart}
\usepackage{amsmath, amsfonts, amssymb,amsthm}
\usepackage{euscript}
\usepackage[T1]{fontenc}
\usepackage{graphicx}
\usepackage{color}

\usepackage{hyperref}

\usepackage{tikz-cd}

\newtheorem{thm}{Theorem}[section]
\newtheorem{lem}[thm]{Lemma}

\newtheorem{prop}[thm]{Proposition}
\newtheorem{cor}[thm]{Corollary}

\theoremstyle{definition}
\newtheorem{defn}[thm]{Definition}

\newtheorem{ex}[thm]{Example}
\newtheorem{rem}[thm]{Remark}

\DeclareMathOperator{\R}{\mathbb R}
\DeclareMathOperator{\C}{\mathbb C}
\DeclareMathOperator{\N}{\mathbb N}

\DeclareMathOperator{\Z}{\mathcal Z}
\DeclareMathOperator{\I}{\mathcal I}

\DeclareMathOperator{\SR}{\mathcal K^0}

\DeclareMathOperator{\SRR}{\mathcal R^0}

\DeclareMathOperator{\Co}{\mathcal C}
\DeclareMathOperator{\SO}{\mathcal O}
\DeclareMathOperator{\Pol}{\mathcal P}
\DeclareMathOperator{\K}{\mathcal K}
\DeclareMathOperator{\QQ}{\mathbb Q}

\DeclareMathOperator{\Cent}{Cent}

\DeclareMathOperator{\RSp}{R-Spec}

\DeclareMathOperator{\Norm}{Norm}

\DeclareMathOperator{\Sp}{Spec}
\DeclareMathOperator{\ReSp}{R-Spec}
\DeclareMathOperator{\MSp}{\rm Max}

\DeclareMathOperator{\p}{\mathfrak{p}}
\DeclareMathOperator{\q}{{\mathfrak q}}
\DeclareMathOperator{\ir}{{\mathfrak r}}

\DeclareMathOperator{\Max}{\rm Max}
\DeclareMathOperator{\ReMax}{\rm R-Max}

\DeclareMathOperator{\m}{\mathfrak m}

\def \S {{\mathcal S}}
\def \T {{\mathcal T}}

\def \RR {{\mathbb R}}

\def\hyp {\textbf{\textrm{\textsc{(mp)}}}}

\begin{document}

\title[\tiny{Integral closures in real algebraic geometry}]
{Integral closures in real algebraic geometry}
\author{Goulwen Fichou, Jean-Philippe Monnier and Ronan Quarez}

\address{Goulwen Fichou\\
Univ Rennes\\
IRMAR (UMR 6625), Campus de Beaulieu, 35042 Rennes Cedex, France}
\email{goulwen.fichou@univ-rennes1.fr}

\address{Jean-Philippe Monnier\\
   LUNAM Universit\'e, LAREMA, Universit\'e d'Angers}
\email{jean-philippe.monnier@univ-angers.fr}

\address{Ronan Quarez\\
Univ Rennes\\
Campus de Beaulieu, 35042 Rennes Cedex, France}
\email{ronan.quarez@univ-rennes1.fr}
\date\today

\begin{abstract} We study the algebraic and geometric properties of
  the integral closure of different rings of functions on a real
  algebraic variety : the regular functions and the continuous rational functions.
\end{abstract}

\maketitle



The normalization of an algebraic variety is obtained by an algebraic
process, corresponding in affine charts to taking the integral closure of the coordinate rings of
the affine components in their rings of fractions. The normalization
of a given algebraic variety $X$ satisfies two universal properties : it is
the biggest algebraic variety with a finite birational morphism onto
$X$ and it is also the smallest normal algebraic variety with a
morphism onto $X$. The normalization can be though of as a kind of
weak polynomial desingularization of
an algebraic variety, but much closer to the original variety due to
the finiteness property.
When dealing with real algebraic varieties however, the normalization procedure may create surprising phenomena, like the lack of surjectivity or the appearance of isolated singular points. The present paper is dedicated to the study of the normalization of real algebraic varieties, together with the integral closure of natural rings of functions on real algebraic varieties, namely the regular functions and the continuous rational functions.

As usual in real algebraic geometry, it is sufficient to
understand the affine case and even to work with real algebraic sets
(which are the real closed points of affine real algebraic varieties)
as explained in \cite{BCR}. A particular difference with affine
algebraic geometry over $\C$ is the fact that polynomial and regular
functions no longer coincide, therefore it is particularly interesting to investigate the relation between the integral closure of polynomial and regular functions. The normalization $X'$ of a
real algebraic set $X$ is obtained classically by the following procedure : the
ring of polynomial functions $\Pol(X')$ on $X'$ is the integral
closure of the ring of polynomial functions $\Pol(X)$ on $X$ in its
total ring of fractions $\K(X)$. We provide in the first section some examples to illustrate some pathological behavior with respect to the topology of the varieties, in relation of the notion of central locus of a real algebraic set.

The regular functions on a real algebraic
set $Y$ are the rational functions with no real poles (e.g. $1/(1+x^2)$ on
$\R$), they form the ring $\SO(Y)$ of regular functions on $Y$ that contains $\Pol(Y)$. These functions are the natural functions when dealing with (the real points of) real varieties, since they forget about the poles lying in the complex part. A significant part of the paper is dedicated to the study of the behavior of
regular functions during the process of normalization. In particular we compare the integral closure of $\SO(X)$ in $\K(X)$ with
$\SO(X')$. We discuss also when the integral closure of the ring of regular functions may be the ring of regular function of a real algebraic set. The material developed there leads to consider real algebraic sets with a totally
real normalization, which correspond to those real algebraic sets for which the rings $\SO(X)$ and $\SO(X')$ coincide. As shown in the third section, many good properties of the normalization for complex algebraic varieties are shared by real algebraic sets with a totally real normalization.

If the normalization of a given algebraic variety $X$ is the biggest algebraic variety finitely birational morphism $X$, one may wonder whether there exists a biggest real algebraic set finitely biregular to a given real algebraic set. To answer positively this question, we investigate in the fourth section the integral closure of $\Pol(X)$ in $\SO(X)$. It is
the ring of polynomial functions of a real algebraic set $X^b$ called
the biregular normalization of $X$. We prove, in particular, that
$X^b$ is the biggest real algebraic set with a finite polynomial maps onto
$X$ which is a biregular isomorphism (i.e that induces an isomorphism
between the rings of regular functions).

Even more flexible than regular functions, we consider also the class of continuous rational functions on a real algebraic set
$X$. A continuous rational functions on $X$ is a continuous function which
is regular on a dense Zariski open subset of $X$ (e.g $x^3/(x^2+y^2)$
on $\R^2$). The concept of continuous rational maps or functions
was used the first time by W. Kucharz \cite{Ku} in order to
approximate continuous maps into spheres. These functions have become
recently the object of a lot of research \cite{FHMM, FMQ, FMQ2, KuKu1,
  KuKu2, KN, Mo}. Koll\'ar \& Nowak \cite{KN} initiated the proper
study of continuous rational functions, proving notably that the
restriction of such a function to an algebraic subset of $X$ does not remain rational in general. It is however
the case as soon as $X$ is nonsingular
\cite[Prop. 8]{KN}. As a
consequence, one may consider the
ring $\SR(X)$ of continuous rational functions, and its subring
$\SRR(X)$ consisting of those continuous rational functions which
remain rational under restriction. This class, called hereditarily
rational in \cite{KN} and regulous in \cite{FHMM}, has been systematically studied in \cite{FHMM}.
It is notably shown in \cite{FHMM} that the use of the sheaf of regulous functions
instead of the sheaf of regular functions corrects some defects of the
classical real algebraic geometry like the absence of an usual
Nullstellensatz and Cartan theorems A and B. In general, we do not have
injections of $\SR(X)$ and $\SRR(X)$ in $\K(X)$ even when $X$ is
irreducible. For that reason, we only consider rational
continuous functions on the central locus $\Cent X$ of an irreducible
real algebraic set $X$. The central locus of an irreducible real
algebraic set is the closure for the Euclidean topology of the set
of smooth points or equivalently it is the locus of points where the
semi-algebraic dimension is maximal. The fifth section is dedicated to
the study of the integral closure of $\SR(\Cent X)$ (the ring of
continuous rational functions on the central locus) in $\K(X)$. In
case $X$ is non-singular, we prove in particular that $\SR(X)$ is
integrally closed. In case $X$ is a curve, we show that the integral
closure of $\SR(\Cent X)$ in $\K(X)$ is $\SO(X')$.

\section{Preliminaries}
In this section we review the basic definition of a real algebraic
variety together with the properties of its normalization. We recall 
the concept of continuous rational functions and hereditarily rational
functions.

\subsection{Real algebraic sets and varieties}
We are interested in this text in the geometry of the real closed points of real algebraic varieties. In this context, it is natural to consider only varieties which are affine since almost all real algebraic
varieties are affine \cite[Rem. 3.2.12]{BCR}. 
We also consider real algebraic sets which are the real closed points of affine real algebraic
varieties. We refer to \cite{Man} for basics on real algebraic
varieties and $\R$-schemes.

More precisely, to a real algebraic variety given by the ideal 
$I$ in $\R[x_1,\ldots,x_n]$ one associates the 
real algebraic set $X=\Z(I)$ of all points in $\R^n$ 
which cancel any polynomial in $I$. 
Conversely, to any real algebraic set $X\subset \R^n$ one may associate
the real algebraic variety given by the ideal $\I(X)\subset \R
[x_1,\ldots,x_n]$ of all polynomials
which vanish at all points of $X$.
Unless specified, all algebraic sets we consider are real.

In complex affine algebraic geometry, polynomial and regular functions
coincide and thus we have a unique and natural definition of morphism
between complex algebraic sets. In the real setting no such natural
definition exists. Indeed, 
the ring of regular functions $\SO(X)$ is the ring of rational
functions with no poles on $X$ (see \cite[Sect. 3.2]{BCR} for details)
which is strictly bigger (if $\dim X>0$) than the ring of polynomial functions $\Pol(X)=\R
[x_1,\ldots,x_n]/I$ where $I=\I(X)$ on a real algebraic set $X$. 

Let $X\subset\R^n$ be a real algebraic set. The complexification of
$X$, denoted by $X_{\C}$, is the complex algebraic set
$X_{\C}\subset\C^n$, whose ring of polynomial functions is
$\Pol(X_{\C})=\Pol(X)\otimes_{\R}\C$. As already mentionned, we have
$\Pol(X_{\C})=\SO(X_{\C})$. Remark that if $X$ is irreducible, then
$X_{\C}$ is automatically irreducible (because $X$ is an algebraic
set).

Let $X\subset\R^n$ and $Y\subset\R^m$ be real algebraic sets. A
polynomial map from $X$ to $Y$ is a map whose
coordinate functions are polynomial. A polynomial map $\varphi:X\rightarrow Y$ induces an $\R$-algebra
homomorphism $\phi:\Pol(Y)\rightarrow \Pol(X)$ defined by
$\phi(f)=f\circ\varphi$. The map $\varphi\mapsto \phi$
gives a bijection
between the set of polynomial maps from $X$ to $Y$ and
the $\R$-algebra homomorphisms from $\Pol(Y)$ to $\Pol(X)$. We say that
a polynomial map $\varphi:X\rightarrow Y$ is an
isomorphism if $\varphi$ is bijective with
a polynomial inverse, or in another words if
$\phi:\Pol(Y)\rightarrow \Pol(X)$ is an
isomorphism. 
We define the analog notions with regular functions in place of polynomials ones. In that situation, an isomorphism will be called a biregular isomorphism. 
Remark that a polynomial map $\varphi:X\rightarrow Y$
extends to a polynomial or regular map $\varphi_{\C}:X_{\C}\rightarrow
Y_{\C}$ but in general a regular map cannot be extended to a regular
map on the complexifications. Unless specified, all the maps are polynomial.

\subsection{Normalization and integral closure}
Let us start with the standard abstract algebraic setting.

Let $A\to B$ be an extension of rings. An
element $b\in B$ is integral over $A$ if $b$ is the root of a monic
polynomial with coefficients in $A$. By \cite[Prop. 5.1]{AM}, $b$ is
integral over $A$ if and only if $A[b]$ is a finite $A$-module. This
equivalence allows to prove that $A_B'=\{b\in B|\,b\, {\rm is\,
  integral\, over}\,A\}$ is a ring called the integral closure of $A$ in
$B$. The extension $A\to B$ is said to be integral if $A_B'=B$. 

Since we will deal with the normalization of (non-necessarily irreducible) algebraic varieties, 
one has to deal with rings $A$ which are not necessarily
domains but only reduced rings.
Hence, in the following $A$ will be a reduced ring admitting 
a finite number of minimal prime ideals and
$B$ be the total ring of fractions $K$ of $A$ (see below), the ring $A_K'$ is
denoted by $A'$ and is simply called the integral closure of $A$.
The ring $A$ is called integrally closed (in $B$) if
$A=A'$ ($A=A_B'$). 

Recall that if $A$ is a reduced ring with minimal prime ideals  
$\p_1,\ldots,\p_r$, then $(0)=\p_1\cap\ldots \cap \p_r$ and one has the canonical injections 
$A\rightarrow A_1\times\ldots\times A_r\rightarrow K_1\times \ldots\times K_r=K$
where $A_i=A/\p_i$ and $K_i$ is the fraction field of $A_i$ for any $i$. 
The product of fields $K$ is called the total ring of fractions of $A$ and 
the $A'_i$'s are called the irreducible components of $A$.

\begin{prop}\label{ReducedIntegralClosure}
Let $A$ be a reduced ring with minimal prime ideals $\p_1,\ldots,\p_r$.
Then, $A'=A'_1\times\ldots\times A'_r$ where 
$A'$ is the integral closure of $A$ in the total ring of fractions $K$, and, for any $i$, $A'_i$ is
the integral closure of $A_i$ (in $K_i$).
\end{prop}
\begin{proof}
Let $f\in K$ such that $P(f)=0$ where $P$ is a unitary polynomial with coefficients 
in $A$. Let us write $f=(f_1,\ldots,f_r)$ where, for each $i$, $f_i\in
K_i$. We may write $Q=Q_1\times\ldots \times Q_r$ with $Q_i$ a unitary polynomial with coefficients in $A_i$.
Then, $P_i(f_i)=0$ in $A_i$ which means that $f_i\in A'_i$.

Conversely, let $f=(f_1,\ldots,f_r)\in K$ be such that each $f_i$ is
integral over $A_i$. Since the idempotents
of $K_1\times \ldots\times K_r$ are also integral over $A$, then one gets that $f\in A'$ and concludes the proof.
\end{proof}

Let us emphasize now what happen in the geometric setting, namely when one 
takes for $A$ the
ring of polynomial functions on an algebraic set $X$ over a
field $k$ ($k=\R$ or $\C$) and when $K$ is total ring of fractions of $X$ denoted in
the following by $\K(X)$. Then, $A'$ is a finite $A$-module (a theorem of Emmy Noether
\cite[Thm. 4.14]{Ei}) and thus $A'$ is a finitely generated $k$-algebra and
so $A'$ is the ring of polynomial functions of an
irreducible algebraic set, denoted by $X'$, called the normalization
of $X$. By Proposition \ref{ReducedIntegralClosure}, we have
$X'=\bigsqcup_{i=1}^r X_{i}'$ where $X_1,\ldots,X_r$ are the
irreducible components of $X$.
We recall that a map $X\rightarrow Y$ between two
algebraic sets over a field $k$ is said finite (resp. birational) if the ring morphism
$\Pol(Y)\rightarrow \Pol(X)$ makes $\Pol(X)$ a finitely generated
$\Pol(Y)$-module (resp. if it induces a biregular isomorphism between
two dense Zariski open subsets of $X$ and $Y$ or equivalently if the
ring morphism $\K(Y)\to \K(X)$ is an isomorphism). For instance, the
projection of the nodal curve given by $y^2-x^2-x^3=0$ on the $x$-axis
is finite but not birational.

The inclusion $A\subset A'$ induces a finite birational map
which we denote by $\pi':X'\rightarrow X$, called the normalization map. We say that an
algebraic variety $X$ over a field $k$ 
is normal if its ring of polynomial functions is integrally
closed.

For a real algebraic set $X\subset \R^n$, we say that $X$ is
geometrically normal if the associated complex algebraic set
$X_{\C}$ is normal. It is well known that $X$ is normal if and only if
$X$ is geometrically normal. 

Note that the normalization of an algebraic set $X$ is the biggest algebraic set finitely birational to $X$. More precisely, for any finite birational map $\varphi:Y\rightarrow X$, there exists $\psi:X'\to Y$ such that $\pi'=\phi \circ \psi$.

Note that finite birational maps behaves nicely with respect to the Euclidean topology.

\begin{prop}\label{lem-closed} Let $\pi:Y\to X$ be a finite birational map between irreducible
algebraic sets. Then $\pi$ is proper and closed for the Euclidean topology. If moreover $\pi$ is injective, then $\pi$ is closed for the constructible topology.
\end{prop}
By constructible topology, we mean the topology which closed sets are given by Zariski constructible sets closed for the Euclidean topology.

\begin{proof}[Proof of Proposition \ref{lem-closed}]
Note that the ring morphism $\Pol(X)\rightarrow \Pol(Y)$ is injective since
$\pi$ is birational. We are going to prove that the map $\pi$ is
closed and proper with respect to the real spectrum topology (see the
section ``Abstract ring of regular functions'' for basics on the real spectrum), then with respect to
the semi-algebraic topology and finally with respect to the Euclidean topology.
By \cite[Ch. 2, Prop 4.2-4.3]{ABR}, the induced map
$\Sp_r\Pol(Y)\rightarrow \Sp_r\Pol(X)$ is closed for the real
spectrum topology. According to \cite[Theorem 7.2.3]{BCR}, there is a 
bijective correspondence between closed semi-algebraic subsets of 
$X$ (resp. $Y$) and closed constructible subsets of the real spectrum $\Sp_r \Pol(X)$
(resp. the real spectrum $\Sp_r \Pol(Y)$). It follows that the image
by $\pi$ of every closed semi-algebraic subset of $Y$ is a closed
semi-algebraic subset of $X$.  
Now it is classical (\cite{vdd} for instance) to conclude 
that $\pi$ is closed and proper for the Euclidean topology.

If $\pi$ is moreover injective, the image by $\pi$ of a Zariski constructible closed subset of $Y$ is a Zariski constructible closed subset of $X$ by \cite[Cor. 4.9]{KP}.
\end{proof}

\subsection{Surjectivity issues and central locus} 
Note that for real algebraic varieties, the normalization map $\pi':X'\to X$
may be non-surjective (consider for instance the cubic $X=\Z(y^2-x^2(x-1))$ with an isolated point at the origin) while the normalization of the complexification $\pi_{\C}':X_{\C}'\to X_{\C}$ is always surjective. 

A similar phenomenon appears in the process of resolving the singularities of a real algebraic varieties. 
We say that a regular map $Y\to X$ is a resolution
of singularities (or a desingularization) if it is a proper birational regular map such
that $Y$ is non-singular. The normalization can be though of as a kind of
weak polynomial desingularization of
an algebraic variety, but much closer to the original variety due to
the finiteness property. For curves, the normalization gives a
polynomial resolution of singularities.

To keep a notion of surjectivity for a resolution of singularities over the real points, we use the concept of central locus of a real algebraic set.

\begin{defn}
\label{centrale}
Let $X$ be an irreducible algebraic set, and denote by $X_{reg}$ the set of non-singular points of $X$. The central locus $\Cent X$ of $X$ is defined to be the Euclidean closure of $X_{reg}$ in $X$.
We say that $X$ is central if $X=\Cent X$.
\end{defn}

\begin{rem}
\label{centraledim}
Let $X$ be an irreducible algebraic set. By \cite[Prop. 7.6.2]{BCR}, $\Cent X$ is
the locus of points of $X$ where the local semi-algebraic dimension is
maximal.
\end{rem}

The central locus well behaves during the process of
desingularization. 
\begin{prop}\cite[Prop. 2.33]{Mo} and \cite[Thm. 2.6, Cor. 2.7]{KK}.\\
\label{centralsurj2}
Let $\pi:Y\to X$ be a resolution of singularities. Then $\pi:\Cent Y=Y\to\Cent X$ is well
defined, surjective and closed for the Euclidean topology.
\end{prop}

Notice that the central locus also well behave under finite birational maps.
\begin{prop}
\label{centralsurj}
Let $\pi:Y\to X$ be a finite birational map between irreducible
algebraic sets. The $\pi_{|\Cent Y}:\Cent Y\to\Cent X$ is well-defined, surjective and closed for the Euclidean topology.
\end{prop}

\begin{proof}
Let $y\in \Cent Y$, and choose a closed connected semi-algebraic neighborhood $N$ of $y$ in $Y$. Then $\pi(N)$ is closed by Lemma \ref{lem-closed}, connected by continuity of $\pi$ and its semialgebraic dimension is equal to $\dim X$ by birationality of $\pi$. Therefore $\pi(y)$ belongs to $\Cent X$. 

Then $\pi(\Cent Y)$ is a closed semialgebraic subset of $\Cent X$, and there exist semialgebraic subsets $A\subset \Cent X$ and $B\subset \Cent Y$ of dimension strictly smaller such that $Y\setminus B$ is in bijection with $X\setminus A$, by birationality of $\pi$. Therefore $\pi(\Cent Y)=\Cent X$.
\end{proof}

These results will be useful is section \ref{sec-cont-rat} when
dealing with rational continuous functions.

Note however that the property of being central is not always preserved by finite and
birational maps. This pathology is illustrated by the following
example where the normalization of a central surface creates an isolated point.

\begin{ex}\label{grospoint2}
Consider the surface $S=\Z((y^2+z^2)^2-x(x^2+y^2+z^2))$ in $\R^3$. Then $S$ is central with a unique singular point at the origin. The singular set of its complexification consists of two complex conjugated curves crossing at the origin. The rational function $f=(y^2+z^2)/x$ satisfies the integral equation $f^2-f-x=0$. Let $Y$ be the surface in $\R^4$ admitting as ring of polynomial function $\Pol(Y)=\Pol(S)[(y^2+z^2)/x]$. We have
$$\Pol(Y)\simeq \dfrac{\R[x,y,z,t]}{((y^2+z^2)^2-x(x^2+y^2+z^2),t^2-t-x,xt-(y^2+z^2),(y^2+z^2)t-(x^2+y^2+z^2))}$$
and since
$(y^2+z^2)/x$ is integral over $\Pol(S)$ we get a finite birational map $\pi:Y\to S$. 
Note that $Y$ may be embedded in $\R^3$ via the projection forgetting the $x$ variable, giving rise to the surface defined by the equation $y^2+z^2=t^2(t-1)$ in $\R^3$. This surface is no longer central, with an isolated singular point at the origin. The preimage of the origin in $S$ consists of two points, the isolated point in $Y$ plus a smooth point in the two dimensional sheet of $Y$. Note that $Y$ is normal since its complexification is an (hyper)surface with a singular point. So $Y$ is the normalization of $S$.
\end{ex}

\section{Integral closure of the ring of regular functions}

We have recalled how the normalization of a variety is obtained via the integral closure of 
the coordinate ring into the 
total ring of fractions. The aim of this section is to study what happens when
one replaces the coordinate ring with the ring of regular functions.

Let $X$ be a real algebraic set.
The ring of regular functions on $X$
is the localization $\SO(X)=\S(X)^{-1}\Pol(X)$ of the ring $\Pol(X)$
of polynomial functions on $X$ with respect to the multiplicative
subset $\S(X)=\{f\in\Pol(X)|\;\Z(f)=\emptyset\}$ of nowhere vanishing
functions \cite[Defn. 3.2.1, Prop. 3.2.3]{BCR}. Note that
non-isomorphic algebraic sets may share isomorphic rings of regular
functions, cf. \cite[Ex. 3.2.8]{BCR}.
Clearly, $\S(X)$ doesn't contain any zero divisor of $\Pol(X)$ (if
$pq=0$ and $\Z(p)=\emptyset$ then $\Z(q)=X$ i.e $q=0$) and
thus we get a natural injective morphism $$\Pol(X)\hookrightarrow \SO(X),\;f\mapsto\dfrac{f}{1}.$$

\subsection{Real prime ideals}

For a commutative ring $A$ containning $\QQ$ we denote by $\Sp A$ the Zariski spectrum of $A$,
the set of all prime ideals of $A$. We denote by $\MSp A$ the set of maximal ideals of $A$. 
In this work, we also consider the real Zariski spectrum $\ReSp A$ which consists
in all the real prime ideals of $A$. The set of maximal and real
ideals of $A$ will be denoted by $\ReMax A$.

Recall that an ideal $I$ of $A$ is called
real if, for every sequence $a_1,\ldots,a_k$ of elements of $A$, then
$a_1^2+\cdots+a_k^2\in I$ implies $a_i\in I$ for $i=1,\ldots,k$.

In the following $A$ will mainly stands for the ring $\Pol(X)$ or $\SO(X)$
and, whatever the considered ring, we denote by 
$\m_x$ the maximal ideal of functions that vanish at $x\in X$. It appears that
in $\SO(X)$ any maximal ideal is real:

\begin{prop}
\label{maxreg1}
We have $\MSp \SO(X)= \ReSp \SO(X).$ 

\end{prop}

\begin{proof}
Assume
$f_1^2+\cdots+f_k^2\in\m$ for some $f_i\in\SO(X)$ and suppose
moreover that $f_1\not\in \m$. If $k=1$ then we get a
contradiction since a maximal ideal is radical. So assume $k>1$. Since $\m$ is maximal,
there exist $g\in \SO(X)$ and $h\in \m$ such that
$gf_1=1+h$. We get $g^2f_1^2+\cdots+g^2f_k^2\in\m$ and
$g^2f_1^2=(1+h)^2=1+h'$ with $h'\in\m$. Hence the invertible function $1+\sum_{i=2}^k
g^2f_i^2$ belongs to $\m$, leading to a contradiction.
\end{proof}

Using then the real Nullstellensatz
\cite[Thm. 4.1.4]{BCR}, one can prove that 
the sets $\MSp \SO(X)$ and $\ReMax \Pol(X)$ are in bijection with $X$. More precisely:

\begin{prop}
\label{maxreg2}
Any maximal ideal of $\SO(X)$ or $\ReSp \Pol(X)$ has the form $\m_x$
for some $x\in X$. Moreover, for $x\in X$ we
have $$\Pol(X)_{\m_x}= \SO(X)_{\m_x}=\SO_{X,x}.$$
\end{prop}

\begin{proof}
The equality $\Pol(X)_{\m_x}= \SO(X)_{\m_x}$ follows from
\cite[Cor. 4, Sect. 4]{Ma}. By \cite[Sect. 3.2]{BCR}, we get
$\Pol(X)_{\m_x}=\SO_{X,x}$.
\end{proof}

Let us set now the algebraic setting which generalizes the geometric viewpoint. 
\subsection{Abstract ring of regular functions}
One may associate to a ring $A$ a topological subspace $\Sp_r A$ 
which takes into account only prime ideals $\p$ whose residual field admits an ordering. Let us detail 
this construction a bit. 


Recall that $\Sp_r A$ is empty if and only if $-1$ is a sum of squares in $A$.  
For any subset $I\subset A$ we define its zero set $\Z_A(I)\subset\Sp_rA$ by 
$\Z_A(I)=\{\p\in\Sp A|\,f\in \p\,\,\forall f\in I\}$. 

Denote by $\T(A)$ the
multiplicative part of $A$ which consists in all elements of $A$ that can be written
as $1$ plus a sum of squares of elements in $A$, a set also denoted by $1+\sum A^2$.

The ring of
regular fractions of elements in $A$ on $\Sp_rA$, denoted by $\SO(A)$,
is defined by
$$\SO(A)=\T(A)^{-1}A.$$ 

Note that, to be more closely related to the geometric case, 
one may also consider $\S(A)=\{f\in A|\,\Z_A(f)=\emptyset\}$.
In general $\T(A)\subset \S(A)$ and the equality does not hold.
Nevertheless, by the Positivestellensatz (\cite[4.4.2]{BCR})
for any element $s\in \S(A)$, there is a sum of squares $u$ and $v$  
in $\T(A)$ such that $us^2=v$, which shows that $\S(A)^{-1}A=\T(A)^{-1}A=\SO(A)$.

Now, to recover the geometric setting, let $A=\Pol(X)$ with $X$ a real
algebraic set. We have $\SO(\Pol(X))=\SO(X)$ since
$\S(\Pol(X))=\S(X)$. Indeed, for $f\in\Pol(X)$ then $\Z_{\Pol(X)}(f)=\emptyset$
if and only if $\Z(f)=\emptyset$ (Artin-Lang Theorem
\cite[Thm. 4.1.2]{BCR}).

\begin{prop}
\label{keyreg}\label{keyregprime}
If $\p\in\RSp A$ is a real prime ideal, then $\S(A)\cap\p=\emptyset$.

Let $\m\in\MSp A$. Then $\m$ is real if and only if $\S(A)\cap\m=\emptyset$.
\end{prop}

\begin{proof}
Assume $\p$ is real. By the real Nullstellensatz \cite[Thm. 2.8]{ABR},
we have $\I(\Z_A(\p))=\p$. If $\S(A)\cap \p\not=\emptyset$ then
$\Z_A(\p)=\emptyset$ and we get a contradiction.

Assume $\m$ is not real. Arguing as in the proof of Proposition
\ref{maxreg1}, there exists $a\in \m$ of the form $1+s$ with $s\in\sum
A^2$. Clearly, $a\in\S (A)$ and thus $\S(A)\cap \m\not=\emptyset$.
\end{proof}

By \cite[Prop. 3.11]{AM} and the above proposition, we get an abstract version 
of Proposition \ref{maxreg1}.
\begin{cor}
\label{maxreg1bis}
We have $\MSp \SO(A)\subset \RSp \SO(A)$. More precisely, any
maximal ideal of $\SO(A)$ is of the form $\S(A)^{-1}\m$ with
$\m\in\ReMax A$. Moreover, any real prime ideal of $\SO(A)$ is of the
form $\S(A)^{-1}\p$ with
$\p\in\ReSp A$.
\end{cor}

By \cite[Cor.4, Ch. 2, Sect. 4]{Ma} and the above proposition, we get:
\begin{cor}
  \label{maxreg1ter}
  Let $\p\in\RSp A$, then $A_{\p}=\SO(A)_{\p\SO(A)}$.
\end{cor}

We recall from \cite[Thm. 4.7, Ch. 2, Sect. 4]{Ma} that, for any integral domain $A$, one has 
$$A=\bigcap_{\p\in\Sp A}A_{\p}=\bigcap_{\m\in\MSp  A}A_{\m},$$
where the intersection has sense in the fraction field of $A$. 

Since all the maximal ideals of a ring of regular functions are real, 
by Corollaries \ref{maxreg1bis} and \ref{maxreg1ter}, one has:
\begin{prop}
\label{equalreg}
Let $A$ be an integral domain. We have $$\SO (A)=\bigcap_{\m\in\ReMax A}A_{\m}.$$
\end{prop}





A result which one can state in the geometric setting as:
\begin{prop} \label{PolyReglocalisation}
Let $X$ be an
irreducible algebraic set. We have
$$\SO(X)=\bigcap_{\p\in\Sp \SO(X)}\SO(X)_{\p}=\bigcap_{x\in X}\SO(X)_{\m_x}=\bigcap_{x\in
  X}\Pol(X)_{\m_x}=\bigcap_{x\in X}\SO_{X,x}.$$ 
\end{prop}

\subsection{Regular functions and normalization}

Let $X$ be a real algebraic set with normalization map
$\pi':X'\to X$. We recall that $\K(X)$ denotes the total ring of
fractions of $X$. Let $X_1,\ldots,X_t$ be the irreducible components of
$X$. They are the zero sets of the minimal prime ideals
$\p_1,\ldots,\p_t$ of $\Pol(X)$. Since these ideals are real, it
follows from the Nullstellensatz \cite[Thm. 4.1.4]{BCR} that they do
not meet $\S(X)$ and thus $\p_1\SO(X),\ldots,\p_t\SO(X)$
are the minimal prime ideals of $\SO(X)$. For $i=1,\ldots,t$, we have
$\Pol(X_i)=\Pol(X)/\p_i$ and $\SO(X_i)=\SO(X_i)/\p_i\SO(X_i)$ since
localization and quotient commute. It follows that
$\K(X)=\prod_{i=1}^t\K(X_i)$ is also the total ring of fractions of $\SO(X)$
and we get the following commutative diagram (all the maps are injective):
\[
\begin{array}{ccccccc}
\Pol(X) & \longrightarrow & \prod_{i=1}^t\Pol(X_i)& \longrightarrow & \prod_{i=1}^t\K(X_i)\\
\downarrow &&\downarrow  & \nearrow& \\
\SO(X) &\longrightarrow &\prod_{i=1}^t\SO(X_i) &  &  \\
\end{array}
\]\\

The goal of this subsection
is to compare the integral closure $\SO(X)'$ of $\SO(X)$ and the ring
$\SO(X')$ of regular functions on the normalization of $X$. By Proposition \ref{normalisationred} below, it will be sufficient to understand the irreducible case.

Note that the process of localization and taking integral closure commute \cite[prop. 5.12]{AM}, 
so that $\SO(X)'$ is equal to $\S(X)^{-1}\Pol(X')$. It is easy to check that if $f\in \S(X)$ then
$f\circ\pi'\in \S(X')$ and thus $$\SO(X)'\hookrightarrow
\SO(X').$$

To generalize to an abstract algebraic setting, one may assume that our ring $A$ 
satisfies the condition:
\vskip0,4cm
\hyp\; The ring $A$ is reduced with a finite number of minimal primes
$\p_1,\ldots,\p_t$ that are all real ideals. 
\vskip0,4cm

This hypothesis fit to our geometric settings where we consider real algebraic 
varieties whose irreducible components are all real.
For instance, we will not consider rings like
$A=\R[x,y]/(x^2+y^2+1)$ nor $A=\R[x,y]/(x^2+y^2)$. 

Under this assumption, it follows from Proposition \ref{keyreg} that $\S(A)$ does not contain any zero divisor
element of
$A$ (the set of zero divisor elements of $A$ is $\cup_{i=1}^t \p_i$) and we get a sequence of injective rings morphisms
$$A\hookrightarrow \SO(A)\hookrightarrow K=\prod_{i=1}^t k(\p_i),$$
where $K$ is the total ring of fractions of $A$ and $k(\p_i)$ is the
residue field at $\p_i$.

It is worth mentionning that 
hypothesis \hyp  $\,$ is preserved by integral extensions contained in the integral closure:
\begin{lem}\label{IntegralExtensionHypothesis}
 Let $A\rightarrow B\rightarrow A'$ be a sequence of ring extensions
 where $A$ is reduced and $A'$ is the integral closure  of $A$.
 
If $A$ satisfies \hyp, then $B$ satisfies also \hyp\, and moreover the
minimal prime ideals of $B$ and $A$ are in bijection given by contraction.
\end{lem}

\begin{proof}
By Proposition \ref{ReducedIntegralClosure} then $A'$ satisfies also
\hyp\, and moreover the
minimal prime ideals of $A'$ and $A$ are in bijection given by
contraction. By $(ii)$ of \cite[Th 9.3, Ch.2, Sect. 9]{Ma}, if $C\to D$ is an
integral extension of rings then a prime ideal of $D$ lying over a
minimal prime ideal of $C$ is a minimal prime ideal of $D$. It follows
now from the lying over property (\cite[Th 9.3 (i), Ch.2, Sect. 9]{Ma}
or Proposition \ref{LOP}) that $B$ has a finite number of minimal
prime ideals which are in bijection with those of $A$ and $A'$. Since
these ideals are contraction of real ideals of $A'$, they are real ideals.
\end{proof}

Let us state now how are related integral closure and ring of regular functions:
\begin{prop}
  \label{abstractnormalisationred}
Let $A$ be a ring satisfying \hyp. 
Denote by $\p_1,\ldots,\p_t$ its minimal primes and $A_i=A/\p_i$.
Then, one has $\SO(A)'\simeq \prod_{i=1}^t\SO(A_i)'$
and $\SO(A')\simeq \prod_{i=1}^t\SO(A_i')$ and a canonical  
injective map $\SO(A)'\rightarrow \SO(A')$.

\end{prop}
\begin{proof}
Let us recall first that \hyp\; implies that the maps $A\rightarrow \SO(A)$ 
and $A'\rightarrow \SO(A')$ are all injective.

Moreover, by Proposition \ref{ReducedIntegralClosure}, one gets 
the isomorphisms 
$\SO(A)'\simeq \prod_{i=1}^t\SO(A_i)'$
and $A'\simeq \prod_{i=1}^t A_i'$ which implies that $\SO(A')\simeq \prod_{i=1}^t\SO(A_i')$. 

Let us summarize our morphims into the following commutative diagramm:
\begin{center}
\begin{tikzcd}[row sep=scriptsize, column sep=scriptsize]
A\arrow[dddd]\arrow[dr] \arrow[rr]&&\prod_{i=1}^t A_i\arrow[dl]\arrow[dddd]\\
&A'\arrow[d]\simeq\prod_{i=1}^t A_i'\arrow[d]&\\
&\SO(A')\simeq\prod_{i=1}^t \SO(A_i')&\\
&\SO(A)'\arrow[u,dashed]\simeq\prod_{i=1}^t \SO(A_i)'\arrow[u,dashed]&\\
\SO(A)\arrow[ur]\arrow[rr]&&\prod_{i=1}^t \SO(A_i)\arrow[ul]\\
\end{tikzcd}
\end{center}

Let us end with the dashed arrow. Namely, if $u\in\SO(A)'$, then 
$u$ is in the total ring of fractions of $A$, and integral over $\SO(A)$.
Then, there exists $v\in 1+\sum A^2$, such that $uv$ is integral over $A$.
Since $v$ is invertible in $\SO(A')$, one gets $u\in \SO(A')$.
\end{proof}

One has an immediate geometric version:
\begin{prop}
  \label{normalisationred}
Let $X$ be a real algebraic set with normalization map
$\pi':X'\to X$. Let $X_1,\ldots,X_t$ be the irreducible components of
$X$. We have
$$\Pol(X)'=\Pol(X')=\prod_{i=1}^t\Pol(X_i)'=\prod_{i=1}^t\Pol(X_i'),$$
$$\SO(X)'=\prod_{i=1}^t\SO(X_i)'$$ and a natural injective map
$$\SO(X)'=\prod_{i=1}^t\SO(X_i)'\hookrightarrow\SO(X')=\prod_{i=1}^t\SO(X_i')\hookrightarrow
\prod_{i=1}^t\K(X_i).$$
\end{prop}

Note that, if $X$ is normal, then $\SO (X)$ is
also integrally closed, and hence the 
rings $\SO(X)'$ and $\SO(X')$ coincide.
The last subsection will give another sufficient condition 
where this remains true. Before, we have to  
give some property of the real counterpart of the well known lying over property
for integral extensions.

\subsection{Real lying over}

We will use several times in the paper the lying over property for
integral extension of rings.
\begin{prop}\cite[Thm. 9.3 (i) and Lem. 2, Ch. 2, Sect. 9]{Ma}.\\
\label{LOP}
Let $A\to B$ be an integral extension of rings.
\begin{enumerate}
\item The associated map $\Sp B\to\Sp A$ is surjective.
\item By the map $\Sp B\to\Sp A$, the inverse image of $\Max A$ 
is $\Max B$.
\end{enumerate}
\end{prop}

As one useful consequence for the sequel, 
when $A\rightarrow B$ is an integral ring extension and $B$ a domain, 
one has $$B=\bigcap_{\m\in\MSp A} B_{\m}.$$

Looking at the normalization of a general irreducible real algebraic
curve,  we see that integral extensions do not necessarily 
satisfy the real lying over property i.e the map $\ReSp
B\rightarrow \ReSp A$ is not necessarily surjective when $A\to B$ is an integral extension of rings.
Indeed, it is enough to consider a cubic with an isolated point at the origin
given for instance by the equation $y^2-x^2(x-1)=0$.

Let $X$ be a real algebraic set. We know by Proposition \ref{maxreg1} that the maximal ideals of the ring of
regular functions $\SO(X')$ of the normalization of $X$ are all real.
We will show in the next proposition that the set of maximal ideals of the
integral closure $\SO(X)'$ of the ring of regular functions
on $X$ corresponds to the maximal ideals of $\Sp\Pol(X')$ lying over
the maximal and real ideals of $\Sp\Pol(X)$. Therefore, a maximal ideal of
$\SO(X)'$ is not necessarily real and consequently the rings
$\SO(X)'$ and $\SO(X')$ can be distinct. It shows that the
the normalization of a real algebraic set cannot be obtained by taking
the integral closure of the ring of regular functions but only
considering the integral closure of the ring of polynomial functions.




\begin{prop}
\label{closint1} Let $X$ be a real algebraic set. 
\begin{enumerate}
\item[1)] The set $\Max \SO(X)'$ is in bijection with the set of maximal ideals of $\Sp\Pol(X')$ 
lying over
the real maximal ideals of $\Pol(X)$.
\item[2)] If $X$ is irreducible, then $$\SO(X)'=\bigcap_{x\in
  X}(\SO_{X,x})'=\bigcap_{\m'\in\Sp\Pol(X'),\;\m'\cap\Pol(X)\in\ReMax\Pol(X)}\Pol(X')_{\m'}$$
whereas $$\SO(X')=\bigcap_{\m'\in \ReMax\Pol(X')}\Pol(X')_{\m'}.$$
\end{enumerate}
\end{prop}

\begin{proof}
Consider the integral extension $\SO(X)=\S(X)^{-1}\Pol(X)\to
\SO(X)'=\S(X)^{-1}\Pol(X')$.
It follows from Propositions \ref{maxreg1}, \ref{maxreg2}
and \ref{LOP} that $\Max\SO(X)'$ is in bijection with the set of maximal
ideals of $\Pol(X')$ that do not meet $\S(X)$ which is the set of maximal
ideals of $\Pol(X')$ that intersect $\Pol(X)$ as maximal ideals
associated to points of $X$.
The last equality in the second statement is a consequence of 
Proposition \ref{PolyReglocalisation}. Notice that if
$\m'\in\Max\SO(X')$ then
$\SO(X')_{\m'}=\Pol(X')_{\m'\cap\Pol(X')}$ (Proposition \ref{maxreg2}). By the previous remark
and Proposition
\ref{maxreg2} we get 
$$\SO(X)'=(\bigcap_{x\in
  X}\SO_{X,x})'=\bigcap_{x\in
  X}(\SO_{X,x})'=\bigcap_{\m'\in\Sp\Pol(X'),\;\m'\cap\Pol(X)\in\ReMax\Pol(X)}\Pol(X')_{\m'},$$ that proves the second statement.
\end{proof}

\begin{ex}\label{CubicIsolatedPoint} Consider the cubic $X=\Z(y^2-x^2(x-1))$ with an isolated point at the origin. Then $\Pol(X')=\Pol(X)[y/x]\simeq \R[x,z]/(z^2-(x-1))$, setting $z=y/x$.
The function $f=1/(1+z^2)=x^2/(x^2+y^2)$ is regular on $X'$. However
$f\not\in \Pol(X')_{\m'}$ for the non-real maximal ideal
$\m'=(1+z^2)$ of $\Pol(X')$ lying over
the (real) maximal ideal $\m=(x,y)$ of the origin in
$\Pol(X)$. Indeed we have $1/f \in\m'$. In particular $f\in \SO(X')\setminus \SO(X)'$.
\end{ex}

\begin{ex} 
\label{ExKollar} Consider the Koll\'ar surface $X=\Z(y^3-x^3(1+z^2))$ \cite{KN}. Then $\Pol(X')=\Pol(X)[y/x]\simeq \R[x,t,z]/(t^3-(1+z^2))$, setting $t=y/x$.
The function $f=1/(t^2+t+1+z^2)=x^2/(y^2+xy+x^2+x^2z^2)$ is regular on
$X'$. Let $\m=(x,y,z)$ be the maximal and real ideal
  of polynomial functions on $X$ that vanish at the origin. Over
  $\m$ we have two maximal ideals of $\Pol(X')$, namely $\m'=(x,t-1,z),\,\m''=(x,t^2+t+1,z)$, and only one of
  these two ideals is real. We have
  $f\not\in\Pol(X')_{\m''}$ since $1/f \in
  \m''$. It follows that $f\in \SO(X')\setminus \SO(X)'$.
\end{ex}

We denote by $\Norm(X_{\C})$ the normal locus of $X_{\C}$ i.e the set of $y\in X_{\C}$ 
such that $\SO_{X_{\C},y}$ is integrally closed.
From Propositions \ref{maxreg2}, \ref{normalisationred} and
\ref{closint1}:

\begin{prop}
\label{intclosreel2}
We have:
\begin{enumerate}
\item[1)] $\Pol(X)$ is integrally closed if and only if $\Norm(X_{\C})=X_{\C}$.
\item[2)] $\SO(X)$ is integrally closed if and only if $X\subset \Norm(X_{\C})$.
\end{enumerate}
\end{prop}

We give in Remark \ref{rem-intclos} an example of a real algebraic set with an integrally closed ring of regular functions, while its ring of polynomial functions is not integrally closed.

\section{Totally real normalization}

As we have seen, the study of the integral closure of the ring of regular functions of a real algebraic set $X$ is highly related to the behaviour of the complexification $\pi_{\C}':X'_{\C}\rightarrow X_{\C}$ of its normalization map $\pi':X'\to X$.  In the sequel, we focus on those real algebraic sets for which the normalization is {\it totally real}. 

\begin{defn}
\label{trn}
We say a real algebraic set $X$ has a totally real normalization if
$\pi_{\C}'^{-1}(X)=X'$. 
\end{defn}

We will see later that real algebraic varieties with a
totally real normalization satisfy very natural properties with respect to finite
birational maps onto them. As an appetizer, we begin with a criterion for the equality between $\SO(X)'$ and $\SO(X')$.

\begin{prop}
\label{egalite}
The rings $\SO(X)'$ and $\SO(X')$ are isomorphic if and only
if $X$ has a totally real normalization.
\end{prop}

\begin{proof}
By Proposition \ref{normalisationred}, it is sufficient to deal with
the case $X$ is irreducible.

Assume there exists $x\in X$ such that $\pi_{\C}'^{-1}(x)$ is not
totally real. It forces the existence of a non-real maximal ideal
$\m'$ of $\Pol(X')$ such that $\m'\cap \Pol(X)=\m_x$. By Propositions \ref{closint1} (first statement) and
\ref{maxreg2} then $\SO(X)'$ and $\SO(X')$ cannot be
isomorphic.

Assume the fibers of $\pi'_{\C}:X'_{\C}\rightarrow X_{\C}$ over the points
of $X$ are totally real. It follows
that $\ReMax\Pol(X')$ is the set of ideals of $\Pol(X')$ lying over
the ideals of $\ReMax\Pol(X)$ and we conclude the proof using the
second statement of Proposition \ref{closint1}.
\end{proof}

Taking the integral closure $\Pol(X)'$ of $\Pol(X)$ in $\K(X)$ is an
algebraic process but it has a geometric counterpart, indeed $\Pol(X)'$
is the ring of polynomial functions of an algebraic set
$X'$. 
We may wonder whether taking the integral closure of $\SO(X)$ into $\K(X)$
admits also a geometric counterpart, namely whether $\SO(X)'$ is the ring of regular functions
of an algebraic set. 

\begin{prop}
\label{clotregular}
If $\SO(X)'$ is the ring of regular functions
of an algebraic set, then $\SO(X)'$ is isomorphic to $\SO(X')$.
\end{prop}

\begin{proof} By assumption $\SO(X)'$ is the ring of regular functions
of an intermediate algebraic set $Y$ between $X$ and
$X'$ i.e $\Pol(X)\subset\Pol(Y)\subset\Pol(X')$.
Since $\SO(X)'=\SO(Y)$ then, by Proposition \ref{maxreg1}, all
the maximal ideals of $\SO(X)'$ are necessarily real. It follows from
Proposition \ref{closint1} that the fibers of $\pi_{\C}':X'_{\C}\rightarrow X_{\C}$ over the points
of $X$ are totally real. By Proposition \ref{egalite}, we get the
proof.
\end{proof}

\begin{rem}\label{rem-intclos} The ring of regular functions of a real algebraic set can be integrally
closed while its ring of polynomial functions is not integrally closed.
 Consider for instance an irreducible algebraic set $X$ of dimension
  one such that $X_{\C}$ is singular and has only non-real
  singularities (e.g $X=\Z(y^2-(x^2+1)^2x)$). Let $X'$ be the
  normalization of $X$. Since for all $x\in X$ the local ring $\SO_{X,x}$ is 
  regular then it follows from 
  Proposition \ref{closint1} that $\SO(X)=\SO(X)'$. By Proposition \ref{egalite}, we see
  that $\SO(X)$ and $\SO(X')$ are isomorphic. Since $X_{\C}$ is
  not normal then $X$ is not normal and thus
  $\Pol(X)$ and $\Pol(X')$ are not isomorphic rings. 
\end{rem}

Having a totally real normalization is a stable property under finite birational map. Even more, it allows to keep surjectivity similarly to the complex setting. More precisely :

\begin{prop}
\label{prestotrealn} Let $\pi:Y\to X$ be a finite birational map between real algebraic sets, and assume that $X$ admits a totally real
normalization. Then $\pi$ is surjective, the fiber $\pi_{\C}^{-1}(x)$ over a real point $x\in X$ is totally real, and $Y$ has also a totally real
normalization.
\end{prop}

\begin{proof}
It is clear that $X$ and $Y$ have the same normalization $X'$. The results follow from the fact that for $y\in Y$, a
point of $X'_{\C}$ lying over
$y$ necessarily belongs to the fiber $(\pi')_{\C}^{-1}(\pi(y))$, which is real by assumption.
\end{proof}

Concerning more algebraic aspects, algebraic sets with totally real normalization give rise to a going-up property for real prime ideals.

\begin{prop}\label{prestotrealn2}
Let $\pi:Y\to X$ be a finite birational map between real algebraic sets, and assume that $X$ admits a totally real
normalization. Then
\begin{itemize}
\item[1)] The map $\Sp\Pol(Y)\to\Sp\Pol(X)$ has totally real fibers over
the real ideals. Namely, if an ideal 
$\q\in\Sp\Pol(Y)$ lies over $\p\in\ReSp\Pol(X)$ then $\q$ is a real ideal.
\item[2)] The integral extension $\Pol(X)\to \Pol(Y)$ satisfies the
  going-up property for the real prime ideals. Namely, given two real
  prime ideals $\p\subset\p'$ of $\Pol(X)$ and $\q\in\ReSp\Pol(Y)$
  lying over $\p$, there exists $\q'\in\ReSp \Pol(Y)$ lying over $\p'$
  such that $\q\subset\q'$.
\end{itemize}
\end{prop}

\begin{proof}
Let us prove 1).
We see $\Pol(X)$ (resp. $\Pol(Y)$) as the subring of
$\Pol(X_{\C})$ (resp. $\Pol(Y_{\C})$) of elements fixed by the
complex conjugation. We also see $X$ (resp. $Y$) as a subset of
points of $X_{\C}$ (resp. $Y_{\C}$) in the same way. In both cases, we denote the complex conjugation with a bar.
By the lying over property (Proposition \ref{LOP}) there is $\q\in\Sp \Pol(Y)$ lying over $\p$.
Assume $\q$ is not a real ideal. 
It follows that the ideal
$\q_{\C}=\q\otimes_{\R}\C$ is no longer a prime ideal in
$\Pol(Y_{\C})$ and there exists
distinct conjugated ideals $\ir,\overline{\ir}\in\Sp\Pol(Y_{\C})$ such
that $\q_{\C}=\ir\cap\overline{\ir}$. On the contrary, we have $\p_{\C}=\p\otimes_{\R}\C\in \Sp\Pol(X_{\C})$ since $\p$ is a
real ideal. We have
integral extensions 
$$\dfrac{\Pol(X_{\C})}{\p_{\C}}\to
\dfrac{\Pol(Y_{\C})}{\ir} ~~~\textrm{~~~and~~~}~~~ \dfrac{\Pol(X_{\C})}{\p_{\C}}\to
\dfrac{\Pol(Y_{\C})}{\overline{\ir}}$$ that induce finite
polynomial maps between irreducible algebraic sets $W\to
V_{\C}$, $\overline{W}\to V_{\C}$ where $V=\Z(\p)$,
$W=\Z_{\C}(\ir)$ and $\overline{W}=\Z_{\C}(\overline{\ir})$. Remark
that $V_{\C}=\Z_{\C}(\p_{\C})$ since $\p$ is a real ideal.
Let $V'=\Z(\q)$. Notice that the integral extension
$$\dfrac{\Pol(X)}{\p}\to\dfrac{\Pol(Y)}{\q}$$ 
does not induce a finite
map $V'\to V$ since $\q$ is different from $\I(V')$ and moreover $V'_{\C}\not=\Z_{\C}(\q_{\C})=W\cup\overline{W}$.
Since we have integral ring extensions, the real algebras
$$\dfrac{\Pol(X)}{\p} ~~~,~~~ \dfrac{\Pol(Y)}{\q}$$ and the complex algebras
$$\dfrac{\Pol(X_{\C})}{\p_{\C}}~~~,~~~\dfrac{\Pol(Y_{\C})}{\ir}~~~,~~~\dfrac{\Pol(Y_{\C})}{\overline{\ir}}$$ have respectively the same
Krull dimension. Since $\p$ is a real ideal, we have $\dim
V=\dim_{\C}V_{\C}=\dim_{\C}W=\dim_{\C}\overline{W}$. We have
$V'=W\cap\overline{W}$ and thus $\dim V'=\dim_{\C}V'_{\C}<\dim
V=\dim_{\C}V_{\C}=\dim_{\C}W=\dim_{\C}\overline{W}$. It follows that
the surjective finite map $W\to V_{\C}$ (which is the restriction of
$\pi_{\C}$ to $W$) has at least a totally
non-real fiber over a point of $V$. 
We get a contradiction using Proposition \ref{prestotrealn} and the proof of 1) is done.

The extension $\Pol(X)\to \Pol(Y)$ is integral and thus satisfies the
going-up property for prime ideals. By 1) any prime ideal of
$\Pol(Y)$ lying over a real prime ideal of $\Pol(X)$ is real, hence $\Pol(X)\to \Pol(Y)$ satisfies the
  going-up property for the real prime ideals and it gives 2).
\end{proof}

We denote by $\overline{A}^Z$ the closure of a set $A$ with respect to
the Zariski topology.

\begin{cor}\label{cor-clo} Let $\pi:Y\to X$ be a finite birational map between real algebraic sets. If $X$ admits a totally real
normalization, then $\pi$ is closed for the Zariski topology.
\end{cor}

\begin{proof}
Let $W$ be an algebraic subset of $Y$. We aim to show
$\pi(W)$ is an algebraic subset of $X$. It suffices to consider the case $W$
irreducible. So there exists $\q\in\ReSp \Pol(Y)$ such that
$W=\Z(q)$. Let $\p=\q\cap\Pol(X)\in\ReSp\Pol(X)$ and $V=\Z(\p)$. We
have $V=\overline{\pi(W)}^Z$. Let $x\in V$ and $\m_x$ be the
corresponding maximal ideal of $\Pol(X)$. By 2) of Proposition \ref{prestotrealn2} there exists a real
prime ideal $\q'$ lying over $\m_x$ such that $\q\subset\q'$. Since
$\Pol(X)\to \Pol(Y)$ is integral then $\q'$ is maximal and thus
corresponds to a point $y\in W$ such that $\pi(y)=x$. Thus $\pi(W)=V$.
\end{proof}

To illustrate how useful can be such a result, we refer to Proposition
\ref{zerointratcont} in the section ``Zero sets of integral continuous
rational functions'' as an
application in relation with continuous rational functions.

We end this section with a particular focus on finite birational map which are moreover bijective. When the target space admit a totally real normalization, it will lead to a particular rigid class of bijective maps which are not necessarily isomorphisms.

In order to state the result, we introduce the notion of hereditarily
birational maps, inspired by \cite{KN}.
Let $\pi:Y\rightarrow X$ be a finite birational map between
algebraic sets. Let $W$ be an irreducible algebraic subset of $Y$ and
let $V$ denote the Zariski closure of $\pi(W)$.
Then the restriction $\pi_{|W}:W\rightarrow V$ is still finite, but not birational in general.

\begin{defn}
\label{defrestriction}
Let $\pi:Y\rightarrow X$ be a finite birational map between 
algebraic sets. We say that $\pi$ is hereditarily birational if for
every irreducible algebraic subset $W\subset Y$, 
the restriction $\pi_{|W}:W\rightarrow V$ is birational where $V$ denote the Zariski closure of $\pi(W)$.
\end{defn}

Note that this condition is far from being automatic, as illustrated by Koll\'ar surface as in Example \ref{ExKollar}. The normalization $\pi':X'\to X$ is even bijective in this example, however its restriction to $W=X'\cap \{x=0\}$ is the projection map from the curve $\{t^3=1+z^2\}\subset \R^2$ onto the $z$-axis whose inverse is not rational.

\begin{thm}
\label{bijtotrealn}
Let $X$ be a real algebraic set with a totally real
normalization. Let $\pi:Y\to X$ be a bijective finite birational map from a
real algebraic set $Y$. Then
\begin{itemize}
\item $\pi$ is an homeomorphism for the Euclidean, constructible and
  Zariski topologies.
\item the map $\ReSp\Pol(Y)\to
\ReSp\Pol(X)$, $\q\mapsto\q\cap\Pol(X)$, is bijective.
\item $\pi$ is hereditarily birational.
\end{itemize}
\end{thm}

\begin{proof}
We reduce the proof to the irreducible case as follows. Note first
that by definition and Proposition \ref{ReducedIntegralClosure}, each irreducible component of $X$ admits a totally real
normalization. Moreover, $X$ and $Y$ have the same number of irreducible components by birationality of $\pi$, with a correspondence between the components induced by $\pi$. Moreover $\pi$ restricted to such a component $Y_1$ of $Y$ is bijective onto the corresponding component $X_1$ of $X$. We prove the surjectivity like this. The finite birational map $\pi_{|Y_1}:Y_1\to X_1$ induces a map $X_1' \to Y_1$ whose composition $X_1' \to Y_1 \to X_1$ coincides with the normalization $X_1'$ of $X_1$. Then, if $x$ is a point in $X_1$, there exist a real point in the normalization of $X_1$ lying over $x$, and its image in $Y_1$ is then a preimage for $x$ by $\pi_{|Y_1}$.

The map $\pi$ is an homeomorphism for the Euclidean and constructible
topologies using Lemma \ref{lem-closed}. By Corollary \ref{cor-clo}, the map
$\pi$ is closed for the Zariski topology and thus it is also an
homeomorphism for the Zariski topology.

Since the map $\Sp\Pol(Y)\to\Sp\Pol(X)$ has totally real fibers over
the real ideals by Proposition \ref{prestotrealn2}, and since $\pi$ is bijective then it follows from the real Nullstellensatz that
the map 
$$\ReSp\Pol(Y)\to
\ReSp\Pol(X)~~~,~~~\q\mapsto\q\cap\Pol(X)$$ is bijective.

It remains to prove the last statement of the theorem. 
Let $W$ be an irreducible algebraic subset of $Y$
and $V=\overline{\pi(W)}^{Z}$. We have to show that the algebraic fields
extension $\K(V)\to \K(W)$ is an isomorphism. We have $W=\Z(\q)$ for a
$\q\in\RSp \Pol(Y)$ and $V=\Z(\p)$ with $\p=\q\cap \Pol(X)\in \RSp
\Pol(X)$. Since $\p$ and $\q$ are real prime ideals then $W_{\C}$ and $V_{\C}$ are irreducible and $W$ (resp. $V$) is Zariski dense in $W_{\C}$ (resp. $V_{\C}$).
By above results, $\q$ is the unique prime ideal of $\Pol(Y)$ lying over $\p$
and $(\pi_{|W})_{\C}:W_{\C}\to V_{\C}$ is a finite map such that $(\pi_{|W})_{\C}^{-1}(V)=W$ and $\pi_{|W}:W\to V$ is a finite bijective map. Assume the algebraic fields extension $\K(V)\to \K(W)$ has degree $d>1$.
There is a non-empty Zariski open subset $U$ of $V_{\C}$ such that
$(\pi_{|W})_{\C}^{-1}(x)$ consists of precisely $d$ points whenever
$x\in U$. Since $U$ meet $V$, then we get a contradiction and the proof
is done.
\end{proof}

Note that without the assumption that $X$ admits a totally real
normalization in Theorem \ref{bijtotrealn}, the restriction of $\pi$
to an irreducible component is no longer bijective in
general. Consider for instance for $X$ the union of the cubic with an
isolated point from Example \ref{CubicIsolatedPoint} with a vertical line passing from that point. The normalization $Y$ of $X$ consist of a line in disjoint union with the normalization of the cubic, and this is bijective to $X$. However, its restriction to the normalization of the cubic is no longer bijective.

\section{Biregular normalization of a real algebraic variety}

\subsection{Geometric biregular normalization}

Let $X$ be a real algebraic set. In this section, we
study the integral closure $\Pol(X)_{\SO(X)}'$ of $\Pol(X)$ in $\SO(X)$.
From the sequence of inclusions $$\Pol(X)\subset \SO(X)\subset \K(X),$$
we see that $\Pol(X)_{\SO(X)}'$ is a finite $\Pol(X)$-module (as a submodule of
the Noetherian $\Pol(X)$-module $\Pol(X')$) and thus is a finitely
generated $\R$-algebra. So, $\Pol(X)_{\SO(X)}'$
is the ring of polynomial functions of a real
algebraic set, denoted $X^{b}$ in the sequel. Remark that $X^{b}$ is an intermediate algebraic set between $X$
and $X'$ i.e $\Pol(X)\subset\Pol(X^{b})\subset \Pol(X')$.
It follows that we have finite birational polynomial maps $\pi^{b}:X^{b}\to X$
and $\phi : X'\to X^{b}$.

We see $\Pol(X^{b})$ as a subring of $\Pol(X')$.
\begin{prop}
\label{defequivgrn} Considering $\Pol(X^{b})$ as a subset of $\Pol(X')$, the following equality holds :
$$\Pol(X^{b})=\{g\in\Pol(X')|\,\exists f\in\SO(X)\,{\rm
    such\,  that}\,f\circ\pi'=g\,\,{\rm in}\,\, \SO(X')\}.$$
\end{prop}

\begin{proof}
An element of the integral
closure of $\Pol(X)$ in $\SO(X)$
is a rational function on $X$ integral over $\Pol(X)$ that can be 
extended as a regular function $f$ on $X$. This integral rational
function gives rise to a polynomial function $g\in\Pol(X')$ and we get
$f\circ \pi'=g$ on $X'$ since $\pi'$ is a regular map and since
$f\circ \pi'$ and $g$ coincide as rational functions on $X'$.
\end{proof}

\begin{rem}\label{rem-bir} In other words, for any polynomial function $g$ on $X^b$, there exists a regular function $f$ on $X$ such that $f\circ \pi'=g \circ \phi$. As a consequence $g$ and $f\circ \pi^b$ coincide on the image of $\phi$ in $X^b$. That image being Zariski dense by birationality of $\phi$, the (polynomial and) regular functions $g$ and $f\circ \pi^b$ coincides actually on the whole $X^b$.
\end{rem}

The interest in $X^b$ comes from the following result, which states that $X^{b}$ is the biggest intermediate algebraic set
between $X$ and $X'$ with the same regular functions as $X$.

\begin{thm}
\label{propunivbiregweak}
Let $X$ be an algebraic
set, $X'$ be its normalization and $X^b$ be the algebraic set such that $\Pol(X^b)=\Pol(X)_{\SO(X)}'$. 

Then $X^{b}$ is the biggest
algebraic set among the 
intermediate algebraic sets $Y$ between $X$ and $X'$ such that the
induced map $\pi:Y\to X$ is a biregular isomorphism. More precisely,
$\pi^{b}$ is a biregular isomorphism and moreover,
given an algebraic set $Y$ and a finite
birational polynomial map $\pi:Y\to X$, then, $\pi$
is a biregular isomorphism if and only if there exists a finite
birational polynomial map $\theta:X^{b}\to Y$ such that
$\pi\circ\theta=\pi^{b}$ (as polynomial maps).
\end{thm}

As a consequence, we called $X^b$ the biregular normalization of $X$.

\begin{proof}[Proof of Theorem \ref{propunivbiregweak}]
We prove first that the polynomial map $\pi^{b}:X^{b}\to X$
is a biregular isomorphism.
Let $y_1,y_2\in X^{b}$ such that $\pi^{b}(y_1)=\pi^{b}(y_2)$. Pick any $g\in\Pol(X^{b})$, let $f\in\SO(X)$ be as in Remark \ref{rem-bir}. Then 
$$g(y_1)=f\circ \pi^b(y_1)=f\circ \pi^b(y_2)=g(y_2),$$
so that $\pi^{b}$ is injective.

To prove the surjectivity of $\pi^b$, take $x\in X$. If $x$ belongs to the image of $\pi'(X')$, then the result is immediate. So assume that there exist conjugated points
$y,\overline{y}\in X'_{\C}$ such that $\pi'_{\C}(y)=\pi'_{\C}(\overline{y})=x$, the complex normalization map being surjective. Pick any $g\in\Pol(X^{b})$, and let $f\in\SO(X)$ be such that $g\circ \phi= f \circ \pi'$ on $X'$ as in Remark \ref{rem-bir}. By regularity of $f$, this equality holds actually on a Zariski dense open subset $U$ of $X'_{\C}$ that contains $\pi'^{-1}(X)$. In particular $g(\phi(y))=g(\phi(\overline{y}))$, so that $\phi(y)=\phi(\overline{y})$ is a real preimage of $x$ in $X^b$. This proves the surjectivity of $\pi^b$, and so its bijectivity.

Suppose $X^{b}\subset \R^n$ and
consider a coordinate function $y_i$ on $X^{b}$ for $i \in
\{1,\ldots,n\}$. We want to prove that the function $f_i=y_i
\circ (\pi^{b})^{-1}$ is regular on $X$. 
However $f_i\circ \pi^{b}$ is polynomial on $X^{b}$, so that, by
Proposition \ref{defequivgrn}, $f_i$
belongs to $\SO(X)$ as required. It follows that the inverse $(\pi^{b})^{-1}$ is
a regular map as expected.

\vskip 2mm

Let $Y$ be an intermediate algebraic set between $X$ and $X'$ such
that the finite birational polynomial map $\pi:Y\to X$ is a biregular
isomorphism. Denote by $\phi$ the map $X'\to Y$. It follows that the
composition by $\pi^{-1}$ gives an injective mapping $\Pol(Y)\subset
\SO(X)$. Hence,
for any $g\in\Pol(Y)$ there exists $f=g\circ\pi^{-1}\in\SO(X)$ such that
$f\circ\pi=g$ on $Y$ and thus $f\circ\pi'=g\circ\phi$ on $X'$. By
Proposition \ref{defequivgrn}, we get $\Pol(Y)\subset\Pol(X^{b})$ and
thus we get an induced polynomial map $\theta:X^{b}\to Y$.

Assume we have finite birational polynomial maps
$\theta:X^{b}\to Y$ and $\pi:Y\to X$ such that
$\pi^{b}=\pi\circ\theta$ (i.e $Y$ is an intermediate algebraic
set between $X$ and
$X^{b}$). We get the following ring extensions $\SO(X)\to \SO(Y)\to
\SO(X^{b})$. Since $\SO(X)\to\SO(X^{b})$ is an isomorphism it
follows that $\pi$ and $\theta$ are biregular isomorphisms.
\end{proof}

\begin{cor}
\label{intbireg}
Let $X$ be an algebraic
set. Let $Y$ be an intermediate algebraic set between $X$ and
$X^{b}$. Then $Y$ is biregularly isomorphic with $X$ and $X^{b}$.
\end{cor}

\subsection{Algebraic biregular normalization}
The aim of this part is to give a general algebraic setting for the notion of biregular normalization,
a setting which will particularly adapted to study the local point of view.

Let us recall our natural assumption :

\hyp\; The ring $A$ is reduced with a finite number of minimal primes
$\p_1,\ldots,\p_t$ that are all real ideals.

Under this assumption, we recall that we get an inclusion $A\subset\SO(A)$.

\begin{defn}
\label{defalgrn}
Let $A$ be a ring satisfying \hyp. The biregular
integral closure of $A$, denoted by $A^{b}$, is $A_{\SO(A)}'$ i.e the integral closure of $A$ in
$\SO(A)$.
\end{defn}

In the geometric setting i.e $A=\Pol(X)$ with $X$ a real
algebraic set then we get $\Pol(X^{b})=\Pol(X)^{b}$.

The following proposition says that the biregular integral closure is obtained, locally, 
by taking the integral closure at any maximal ideal which is not real, and by doing nothing 
otherwise :
\begin{prop}
\label{equivdefalgrn}
Let $A$ be a a ring satisfying \hyp. We have
\begin{enumerate}
\item[1)] Let $\m\in\MSp A$ but
  $\m\not\in\RSp A$. Then $$(A^{b})_{\m}=(A_{\m})'=(A')_{\m}.$$
\item[2)] Let $\m\in\ReMax A$. Then $$(A^{b})_{\m}=A_{\m}=\SO(A)_{\m}.$$
\end{enumerate}
\end{prop}

\begin{proof}
The proof uses the sequence of inclusions
$A\subset\SO(A)\subset K$ already described where $K$ denotes the
total ring of fractions of $A$.

Let $\m\in\Max A$. Let $\overline{A_{\m}}$ be the integral
closure of $A_{\m}$ in $\SO(A)_{\m}$. 
Since integral closure commutes with localization, we get $\overline{A_{\m}}=(A^{b})_{\m}$. 

Assume $\m\in\MSp A$ but
  $\m\not\in\RSp A$. By Proposition \ref{keyreg},
  $\S(A)\cap\m\not=\emptyset$ and thus $\S(A)^{-1}\m=\SO(A)$
  (\cite[Prop. 4.8]{AM}). It follows that
$\SO(A)_{\m}=K_{\m}$ and thus
$\overline{A_{\m}}=(A^{b})_{\m}=(A_{\m})'=(A')_{\m}$.

Assume $\m\in\ReMax A$. It follows from Corollary \ref{maxreg1ter} that
$\SO(A)_{\m}=A_{\m}$ (thus $(A^{b})_{\m}=A_{\m}$) and the proof of 1)
and 2) is done.
\end{proof}

It follows from Proposition \ref{equivdefalgrn} that biregular
normalization and localization by a maximal ideal commute.
\begin{cor}
Let $A$ be a a ring satisfying \hyp and let $\m\in\MSp
A$. Then $$(A^b)_{\m}=(A_{\m})^b.$$
\end{cor}

If one assume that $A$ is a domain, one deduces from 
$A^b=\bigcap_{\m\in\MSp A} A^b_{\m}$, that:
\begin{prop}
Let $A$ be a real ring which is a domain. We have
$$A^{b}=(\bigcap_{\m\in\MSp
  A\setminus\ReMax A}(A')_{\m})\cap(\bigcap_{\m\in\ReMax
  A}A_{\m})$$
where the intersection has a sense in the fraction field of $A$.
\end{prop}

Let $X$ be a real algebraic set. From the previous
proposition, we see that $X^{b}$ is a normalization of
the non-real locus of $X_{\C}$. 

\begin{ex} 
Consider the irreducible algebraic curve $X=\Z(y^2-(x^2+1)^2x)$). Let $X'$ be the
  normalization of $X$. Since the singularities of $X$ are non-real we
  get $X^{b}=X'$. Remark that $\Pol(X')=\Pol(X)[y/(1+x^2)]$ and that
  $y/(1+x^2)$ is a regular function on $X$ integral over $\Pol(X)$.
\end{ex}

While usual normalization always separates the irreducible components (Proposition \ref{ReducedIntegralClosure}),
it is not the case for the biregular normalization since $\pi^b:X^b\to
X$ is a bijection (Theorem \ref{propunivbiregweak}).

\begin{ex}
Consider the algebraic curve $X=\Z(xy)$. Then $X$ is its own biregular
normalization, while $X'$ is the union of two lines.
\end{ex}

\begin{ex} 
Consider the Koll\'ar surface $X=\Z(y^3-x^3(1+z^2))$
(Example \ref{ExKollar}). Then $\Pol(X')=\Pol(X)[y/x]$.
Since the rational fraction $y/x$ has a pole along the $z$-axis, then
it follows that $X^b\not=X'$ even if $\pi':X'\to X$ is a bijection. By
Proposition \ref{equivdefalgrn}, we see that $X^b=X$.
\end{ex}

The reader is reffered to \cite{FMQ2} to find comparisons of the
biregular normalization with other kinds of normalizations.

\subsection{Biregular extensions of rings}
Let $A\to B$ be an integral extension with $A$ satisfying \hyp\, and $B$ a subring of $A'$, the integral closure of $A$.  
According to Lemma \ref{IntegralExtensionHypothesis}, the ring
$B$ also satisfies \hyp. Since the image of $\T(A)$ by
the extension is contained in $\T(B)$ and since $\T(A)$
(resp. $\T(B)$) doesn't contain any zero divisor element of $A$ (resp.
$B$) then we get an induced ring extension
$\SO(A)\to\SO(B)$ that is non necessarily integral. In case $\SO(A)\to\SO(B)$ is an isomorphism we say that $A\to B$ is biregular.

\begin{rem}
 Another possible framework would be to consider integral extensions $A\to B$ 
 such that $A$ and $B$ both satisfy \hyp.
 As we just said, this condition is fullfiled when $A$ satisfies \hyp\; and $B$ injects into $A'$. The converse being false as one can check with the 
 ring extension $\RR[x]\rightarrow  \RR[x,\sqrt{x}]$.
 
 Nevertheless, the assumption that $B$ injects into $A'$ seemed more natural to us. Moreover, the results we are interested in (Proposition \ref{CharactBiregExt} and Theorem \ref{propunivbiregweakabstract})
 remain exactly the same if we replace one framework with the other.
\end{rem}

Let us give a characterization for biregular extensions. 
\begin{prop}\label{CharactBiregExt}
\label{equibireg}
Let $A$ be a ring satisfying \hyp \; and $A\to B$ an integral extension of rings contained in $A'$. The following
properties are equivalent:
\begin{enumerate}
\item[1)] The extension $A\to B$ is biregular.
\item[2)] Given any  ideal $\m\in\ReMax A$, there exists a unique
maximal $\m'\in\Max B$ lying over $\m$ and we have $\m'\in\ReMax B$ and the
map $A_{\m}\to B_{\m'}$ is an isomorphism. 
\item[3)] Given any real prime ideal $\p\in\RSp A$, there exists a unique
prime $\q\in\Sp B$ lying over $\p$ and we have $\q\in\RSp B$ and the
map $A_{\p}\to B_{\q}$ is an isomorphism. 
\end{enumerate}
\end{prop}

\begin{proof}
Obviously 3) implies 2).

Let us show that 2) implies 1).
Given any ideal $\m\in\ReMax A$, there exists a unique
maximal $\m'\in\Max B$ lying over $\m$ and we have $\m'\in\RSp B$ and the
map $A_{\m}\to B_{\m'}$ is an isomorphism. By \cite[Ex. 3, Ch. 3,
Sect. 9]{Ma}, we get $B_{\m}=B_{\m'}$. Hence we have
$A_{\m}=B_{\m}$, and by Corollary \ref{maxreg1ter} we obtain $A_{\m}=\SO(A)_{\m}=B_{\m}=\SO(B)_{\m}=B_{\m'}=\SO(B)_{\m'}$.
One has $\SO(A)_{\m}=\SO(B)_{\m}$ for any $\m\in\ReMax A$, which shows that 
$\SO(A)_{\m}=\SO(B)_{\m}$ for any $\m\in\Max \SO(A)$.
By \cite[Prop. 3.9]{AM}, we
get that $\SO(A)\to\SO(B)$ is bijective and thus a ring isomorphism.

Let us show now that 1) implies 3).
Assume $A\to B$ is biregular. Let $\p\in\RSp A$ and let
$\q\in\Sp B$ lying over $\p$. By
Corollary \ref{maxreg1ter}, we have $A_{\p}\simeq
\SO(A)_{\p\SO(A)}$. Since $\SO(A)\simeq \SO(B)$ then it follows that
$\q$ does not meet $\T(B)$. Indeed, by the contrary, let us assume that $b\in\q\cap 
\T(B)$. In $\SO(B)=\SO(A)$, one has $b= \frac{a}{1+\sum_ia_i^2}$
where $a$ and the $a_i$'s ly in $A$. One gets $a\in \q\cap A=\p$. Moreover, 
since $b\in\T(B)$, it is invertible in $\SO(A)$. 
Since $\sum_ia_i^2\in \SO(A)$, one gets also that 
$a\in \SO(A)$, a contradiction.

Now, one may look at $\q$ as a prime ideal of $\SO(B)$ lying over the prime ideal $\p$ viewed in 
$\SO(A)$. The isomorphism $\SO(B)=\SO(A)$ shows that $\q$ is unique and also real.
Moreover, one has $$A_{\p}\simeq
\SO(A)_{\p\SO(A)}\simeq \SO(B)_{\q\SO(B)}\simeq B_{\q}.$$
\end{proof}

Note that the previous 
proposition shows that a biregular ring extension is a centrally subintegral ring extension
as defined in \cite{FMQ2}.

We prove now an abstract version of Theorem \ref{propunivbiregweak}, 
which says that the biregular integral closure
$A^{b}$ of $A$ is the biggest intermediate ring 
between $A$ and $A'$ which is biregular with $A$. Namely:

\begin{thm}
\label{propunivbiregweakabstract}
Let $A$ be a ring satisfying \hyp. Let $A'$ denote the integral
closure of $A$. Then,
\begin{enumerate}
 \item $A\to A^{b}$ is biregular, 
 \item Let $A\to B\to A'$ be a sequence of ring extensions.
Then, $A\to B$ is biregular if and only if $B\subset A^{b}$.
\end{enumerate}
\end{thm}

\begin{proof}
Note that, by Proposition \ref{ReducedIntegralClosure}, $A'$ satisfies \hyp. 

We begin by proving $A\to A^{b}$ is biregular. Let
$\m\in\ReMax A$. By Proposition
\ref{equivdefalgrn} we have shown that $(A^{b})_{\m}=A_{\m}$. It
follows that $(A^{b})_{\m}$ is local and thus there exists a unique
ideal $\m'$ (which is a real ideal) of $A^{b}$ lying over $\m$. By \cite[Ex. 3, Ch. 3,
Sect. 9]{Ma}, we get $(A^{b})_{\m}=(A^{b})_{\m'}$ and hence $A_{\m}=(A^{b})_{\m'}$. By Proposition
\ref{equibireg}, 
$A\to A^{b}$ is biregular.

Assume first that $A\to B$ be a biregular extension where $B$ is a subring of $A'$,
and let us show that $B\subset A^{b}$. Since $B$ satisfies \hyp\, (Lemma \ref{IntegralExtensionHypothesis}), $B$ injects into $\SO(B)$.
Since 
$\SO(A)\simeq\SO(B)$, we get an injective map $B\to
\SO(A)$. Since $A\to B$ is integral and since $A^{b}$ is the
integral closure of $A$ in $\SO(A)$ we get the desired inclusion
$B\subset A^{b}$.

Assume now that $A\subset B\subset A^{b}$. It induces $\SO(A)\subset
\SO(B)\subset \SO(A^{b})$. Since $\SO(A)=\SO(A^{b})$, one has 
$\SO(A)=\SO(B)$ and the
proof is done.
\end{proof}

\section{Integral closure of the ring of continuous rational functions}\label{sec-cont-rat}

The intriguing behaviour of rational functions on a real algebraic set admitting a continuous extension to the whole algebraic set has been investigated in \cite{KN}. Among them, the class of hereditarily rational functions is of special interest. We investigate in this section the integral closure of the ring of continuous rational functions on the central locus, and end the section with discussing fully the case of curves.

We start with recalling the main definitions for continuous rational functions and hereditarily rational functions, and focus later on the restriction of these functions to the central locus on the real algebraic set considered.

\subsection{Continuous rational functions and hereditarily rational functions}

For $X\subset \R^n$ an
algebraic set, the total ring of fractions $\K(X)$ is also the ring of
(classes of) rational
functions on $X$ (which is a field when $X$ is irreducible).
A rational function $f\in \K(X)$ is regular on a Zariski-dense open
subset $U\subset X$ if there exist polynomial functions $p$ and $q$ on
$\R^n$ such that $\Z(q) \cap U=\emptyset$ and $f=p/q$ on $U$. The
couple $(U,f_{|U})$ is called a regular presentation of $f$.

\begin{defn}
\label{defratcont}
Let $f:X\to \R$ be a continuous function. We say that
$f$ is a continuous rational function on $X$ if there exists a
Zariski-dense open subset $U\subset X$ such that $f_{|U}$ is
regular. We denote by $\SR(X)$ the ring of continuous rational functions on $X$.

A map $Y\to X$ between real algebraic sets $X\subset \R^n$ and $Y\subset \R^m$ is called continuous rational if its
components are continuous rational functions on $Y$.
\end{defn}

A typical example is provided by the function defined by $(x,y)\mapsto
x^3/(x^2+y^2)$ on $\R^2$ minus the origin, with value zero at the
origin. 
Note also that on a curve with isolated points, a function regular on the one-dimensional branches can be extended continuously by any real value at the isolated points. In particular, the natural ring morphism
$\SR(X)\rightarrow \K(X)$ which sends $f\in \SR(X)$ to the
class $(U,f_{|U})$ in $\K(X)$, where $(U,f_{|U})$ is a regular
presentation of $f$, is not injective in general. 
This phenomenon is related to the notion of central part of a real algebraic set
(see \cite[Prop. 2.15]{Mo}). 

Another stringy phenomenon is illustrated by Koll\'ar example
(Example \ref{ExKollar}). Consider the surface $S=\Z(y^3-(1+z^2)x^3)$ in $\R^3$. The
continuous function defined by $(x,y,z)\mapsto ~^3 \sqrt{1+z^2}$ is
regular on $S$ minus the $z$-axis, however its restriction to the
$z$-axis is no longer rational. This pathology leads to the notion of
hereditarily rational function in the sense of \cite[Def. 1.4]{KuKu2}.

\begin{defn}
Let $X$ be an algebraic set. A continuous rational function $f\in
\SR(X)$ is hereditarily rational on $X$ if for every
irreducible algebraic subset $V\subset X$ the restriction $f_{|V}$ 
is rational on $V$. We denote by $\SRR(X)$ the ring of hereditarily rational functions on $X$.

A map $Y\to X$ between real algebraic sets $X\subset \R^n$ and $Y\subset \R^m$ is called hereditarily rational if its
components are hereditarily rational functions on $Y$. 
\end{defn}

In particular, in the case of curves, the rings $\SR(X)$ and 
$\SRR(X)$ coincide. 
It is known that for a real algebraic set $X$ with at most isolated
singularities, any continuous rational function is also hereditarily
rational \cite{KN,Mo}. Note also that the regulous functions
introduced in \cite{FHMM} on a general real algebraic set $X\subset
\R^n$ as the quotient of $\SRR(\R^n)$ by the ideal of continuous
rational functions vanishing on $X$, coincide with the hereditarily rational functions on $X$. 
By \cite[Thm. 6.4]{FHMM}, the topology generated by zero sets of hereditarily rational functions on $X$ is the algebraically constructible topology on $X$ (denoted by $\Co$-topology).

\begin{defn}
Let $X$ be an irreducible algebraic set. The subring of all continuous functions on $\Cent X$ that are rational on
$X$ is denoted by $\SR(\Cent X)$. A continuous rational function $f\in
\SR(\Cent X)$ is hereditarily rational on $\Cent X$ if for every
irreducible algebraic subset $V\subset X$ satisfying
$V={\overline{V\cap \Cent X}}^Z$, the restriction $f_{|V\cap \Cent X}$ 
is rational on $V$. We denote by $\SRR(\Cent X)$ the ring of
hereditarily rational functions on $\Cent X$.
\end{defn}

The main interest of restricting continuous rational functions to the central locus is that the canonical maps
$\SR(\Cent X)\rightarrow \K(X)$ and $\SRR(\Cent X)\rightarrow
\K(X)$ are now injective.

\subsection{Zero sets of integral continuous rational functions}
The zero set $\Z(f)$ of an hereditarily rational function $f\in\SRR(X)$ is a Zariski constructible set \cite{FHMM}. However in general, the zero set $\Z(f)$ of an arbitrary continuous rational function $f\in\SR(X)$ has only the structure of a closed semi-algebraic set, and it happens indeed that $\Z(f)$ is not Zariski
constructible \cite[Ex. 3.6]{Mo}.

We provide below sufficient conditions on $X$ and $F$ to ensure that such a pathological behavior can not happen. It consist in an application of the notion of totally real normalization introduced in this paper. More precisely, consider a central irreducible real algebraic set $X$, so that it comes with an extension of domains $\Pol(X)\subset\SR(X)$. Then the zero set of functions belonging to the integral closure $\Pol(X)_{\SR(X)}'$
of $\Pol(X)$ in $\SR(X)$ is Zariski constructible, provided that $X$ admits a central and totally real normalization.

\begin{prop}
\label{zerointratcont}
Let $X$ be an irreducible central real algebraic set with a central normalization which is totally real. The zero set of a continuous rational function on $X$,
which is integral over the polynomial functions, is Zariski closed. If
$f$ is assume moreover to be hereditarily rational, then the sign of $f$ on
$X$ coincides with the sign of a polynomial function on $X$.
\end{prop} 

\begin{proof}
Let $\pi':X'\to X$ be the normalization map. By Proposition
\ref{centralsurj} the map $\pi'$ is surjective.
Let $f\in\Pol(X)_{\SR(X)}'$. Since $X$ is central then
$\SR(X)\subset\K(X)$. Let $(U,f_{|U})$ be a regular presentation of
$f$ that is integral over $\Pol(X)$. It gives rise to a function $g\in\Pol(X')$
such that $g=f_{|U}\circ\pi'$ on $\pi'^{-1}(U)$. Since
$X'$ is central, the continuous functions $g$ and $f\circ\pi'$ coincide
on $X'$. It follows that $\Z(f\circ\pi')$ is a Zariski closed subset of
$X'$. It follows from Corollary \ref{cor-clo} that
$\Z(f)=\pi'(\Z(f\circ\pi'))$ is Zariski closed.

The second assertion is a direct consequence of \cite[Thm. B]{Mo}.
\end{proof}






\subsection{Integral closure of the ring of continuous rational
  functions on the central locus}

Let $X$ be an irreducible real algebraic set. Recall that the ring $\SR(X)$ is not always a domain, contrarily to $\SR(\Cent X)$ whose fraction field is isomorphic to $\K(X)$.

Following the arguments used in the proof of \cite[Prop. 2.33]{Mo}, we
can show that the continuous rational functions on the central locus
of $X$ are exactly the functions on $\Cent X$ that become regular
after a well chosen resolution of singularies of $X$ 
(they are called blow-regular functions in \cite{Mo}).
\begin{prop}
\label{blowreg}
Let $f:\Cent X\to \R$ be a real function. The following properties are equivalent:
\begin{enumerate}
\item $f\in\SR(\Cent X)$.
\item There exists a resolution of singularities $\pi:Y\to X$ such that $f\circ\pi\in\SO(Y)$.
\end{enumerate}
\end{prop}

\begin{rem}
Let $\pi:Y\to X$ be a resolution of singularities.
\begin{enumerate}
\item We have $\pi(Y)=\Cent X$ by Proposition \ref{centralsurj2}.
\item The equality $f\circ\pi\in\SO(Y)$ in the statement of the
  previous proposition means that there exists $g\in\SO(Y)$ such that
  $f\circ\pi=g$ on $Y$. In particular, since $\pi$ is a regular birational map, the functions
  $f$, $f\circ\pi$ and $g$ 
  represent the same class in $\K(X)$.
\end{enumerate}
\end{rem}

As a first result, we prove that the ring of continuous rational
functions on a smooth central locus of an algebraic set (or the ring of hereditarily rational functions, it is the same in that case) is integrally closed.

\begin{thm}
\label{intcloslisse1}
Let $X$ be an irreducible algebraic set such that $\Cent(X)=X_{reg}$ (it is equivalent to say that $X_{reg}$ is closed for the Euclidean topology). Then
$\SR(\Cent X)$ is integrally closed in $\K (X)$.
\end{thm}

\begin{proof}
Let $f\in \K(X)^{*}$ be a non-zero rational function on $X$, and assume that $f$ satisfies an integral equation with continuous rational coefficients on $\Cent X$, namely there exist $d\in\N^{*}$ and $a_i\in\SR(\Cent X)$,
$i=0,\ldots, d-1$ such that 
$$f^d+a_{d-1}f^{d-1}+\cdots+a_0=0$$ in $\K (X)$.
Then there exists a non-empty Zariski open subset $U$ of $\Cent X$ 
such that 
$$\forall x\in U,~~f^d(x)+a_{d-1}(x)f^{d-1}(x)+\cdots+a_0(x)=0.$$ 
By Proposition \ref{blowreg}, there exists a resolution of singularities $\pi:Y\rightarrow X$ such that 
$\tilde{a_i}=a_i\circ \pi$ is regular on $Y$ for $i=0,\ldots, d-1$. 
In particular
$f\circ \pi$ is a rational function on $Y$ which is integral
over $\SO (Y)$. Since $Y$ is non-singular then $\SO (Y)$
is integrally closed in $\K(Y)$ by Proposition \ref{intclosreel2}, and thus the rational function $f\circ \pi$ can be extended to a
regular function $\tilde{f}$ on $Y$. Obviously, we have 
$$\forall y\in Y,~~\,\tilde{f}^d(y)+\tilde{a}_{d-1}(y)\tilde{f}^{d-1}(y)+\cdots+\tilde{a}_0(y)=0.$$ 
Let $x\in \Cent X=X_{reg}$. Since each regular function $\tilde{a}_i$ is
constant on $\pi^{-1}(x)$, then for all $y\in\pi^{-1}(x)$ the real
number $\tilde{f}(y)$ is a root of the one variable polynomial

$$p(t)=t^d+\tilde{a}_{d-1}(x)t^{d-1}+\cdots+\tilde{a}_0(x)\in\R[t].$$ 
Since
$\pi^{-1}(x)$ is connected in the Euclidean topology and moreover $\tilde{f}$ is continuous on $Y$, we obtain that
$\tilde{f}$ must be constant on $\pi^{-1}(x)$. Hence
$\tilde{f}$ induces a real continuous function $f^c$ on $\Cent X$ such that
$\tilde{f}=f^c\circ \pi$ ($\pi:Y\to \Cent X$ is a quotient map for the Euclidean topology \cite[Prop. 2.33]{Mo}) and therefore $f^c$ is a continuous extension to $\Cent X$ of $f$.
\end{proof}

\begin{cor}
\label{intcloslisse}
Let $X$ be a non-singular irreducible algebraic set. Then
$\SR(X)$ is integrally closed in $\K (X)$.
\end{cor}

A rational function which does not admit a continuous extension on a given algebraic set
may admit different behaviors at a indeterminacy point. 
It can be unbounded like $1/x$ at the origin in $\R$, bounded
with infinitely many limit points like $x^2/(x^2+y^2)$ at the
origin in $\R^2$, 
or bounded with finitely many limit points like in the case of rational function satisfying an integral equation with continuous rational coefficients on the central locus.

\begin{lem} 
\label{indetfinite} 
Let $X$ be an irreducible
  algebraic set. Assume $f\in \K(X)$ satisfies an integral equation
  with coefficients in $\SR(\Cent X)$. Then $f$ admits finitely many different limits at its indeterminacy points that are central.
\end{lem}

\begin{proof} 
The rational function $f$ satisfies an integral equation of the form $$f^d+a_{d-1}f^{d-1}+\cdots+a_0=0$$
with $d\in\N^{*}$ and  $a_i\in\SR(\Cent X)$, for $i=0,\ldots, d-1$. Let
$\pi:Y\to X$ be a resolution of the singularities of $X$. The functions $a_i\circ \pi$ are continuous rational functions on a non-singular variety, therefore they are hereditarily rational on
$Y$. Then $f\circ \pi$ is a rational function on $Y$ satisfying an integral
equation with hereditarily rational coefficients. By Corollary \ref{intcloslisse},
$f\circ\pi$ can be extended to $Y$ as an hereditarily rational function $\tilde{f}$
on $Y$. Let $x\in \Cent X$. By the arguments
used in the proof of Theorem \ref{intcloslisse1} then $\tilde{f}$ is
constant on the connected components of $\pi^{-1}(x)$. Since $\tilde{f}$ is hereditarily rational on $Y$, it follows that
$\tilde{f}|_{\pi^{-1}(x)}$ is hereditarily rational on $\pi^{-1}(x)$. In particular
$\tilde{f}$ is constant on each $\Co$-irreducible component of
$\pi^{-1}(x)$. Since there exists a finite number of such components, $\tilde{f}$ takes a finite numbers of values on $\pi^{-1}(x)$, and therefore $f$ admits finitely many limit points at $x$.
\end{proof}

We extend the notion of connectedness to the constructible topology.

\begin{defn}
Let $X$ be an
  algebraic set. Let $Y_1,\ldots,Y_k$ be the $\Co$-irreducible
  components of $X$. We say that $X$ is $\Co$-connected if
  $k=1$ or else if $\forall i\not= j$ in $\{1,\ldots,k\}$ there exists a sequence
  $(i_1,\ldots,i_l)$, $l\geq 2$, of two by two distinct numbers in
  $\{1,\ldots,k\}$ such that $i_1=i$, $i_l=j$, and for $t=1,\ldots,
  l-1$ then $Y_{i_t}\cap Y_{i_{t+1}}\not=\emptyset$.
\end{defn}

For example, an algebraic set $X$ is 
$\Co$-connected when $X$ is connected. This notion enables to extend Theorem \ref{intcloslisse1} to certain singular cases.
\begin{prop}
\label{intclossing}
Let $X$ be an irreducible algebraic set such that there
exists a resolution of singularities $\pi:\tilde{X}\rightarrow X$ such
that for all $x\in \Cent X$ the fiber $\pi^{-1} (x)$ is $\Co$-connected. Then
$\SR(\Cent X)$ is integrally closed in $\K (X)$.
\end{prop}

\begin{proof}
Assume $f\in \K(X)^{*}$ there exist $d\in\N^{*}$ and  $a_i\in\SR(\Cent X)$,
$i=0,\ldots, d-1$ such that 
$$f^d+a_{d-1}f^{d-1}+\cdots+a_0=0$$ in $\K (X)$. Let
$\pi:\tilde{X}\rightarrow X$ be a resolution of singularities such
that $\forall x\in \Cent X$ the fiber $\pi^{-1} (x)$ is $\Co$-connected.
As we have already explained in the proof of Lemma \ref{indetfinite},
the rational function $f\circ \pi$ can be extended as an hereditarily
rational function to
$\tilde{X}$. Let $\tilde{f}\in\SRR(\tilde{X})$ denote the
extension. Let $x\in \Cent X$. We know that $\tilde{f}$ is constant on the
connected components of $\pi^{-1}(x)$. Let $Y_1,\ldots,Y_k$ be the $\Co$-irreducible
  components of $\pi^{-1}(x)$. Since for $i=1,\ldots,k$
  $\tilde{f}|_{Y_i}\in\SRR(Y_i)$ (see \cite[Cor. 5.38]{FHMM}) then
  $\tilde{f}$ is constant on $Y_i$ (see \cite[Cor. 6.6]{FHMM}). Since
  $\pi^{-1}(x)$ is $\Co$-connected then $\tilde{f}$ is constant
  on $\pi^{-1}(x)$. We conclude the proof in the same way we did in
  the proof of Theorem \ref{intcloslisse1}.
\end{proof}

\begin{ex}\begin{enumerate}
\item Let $X$ be the cuspidal plane curve given by $y^2-x^3=0$. By Proposition
  \ref{intclossing} we know that $\SR(X)$, which is also equal to $\SRR(X)$ since $X$ is a curve, is integrally
  closed.
  \item Let $X$ be the nodal plane curve given by $y^2-(x+1)x^2=0$. The
  rational function $f=y/x$ is integral over $\Pol (X)$ since
  $f^2-(x+1)=0$ on $X\setminus\{ (0,0)\}$. It is easy to see that $f$
  cannot be extended continuously to whole $X$. Hence
  $\SR(X)$ is not integrally closed. Of course the fiber
  over the node cannot be connected when we resolve it.
\item Let $X$ be the central algebraic surface in $\R^3$ defined as by
  $y^2=(z^2-x^2)(x^2-2z^2)$. It can be view as the cone over the
  non-singular curve defined in the plane $z=1$ by the irreducible curve 
with two connected components $y^2=(1-x^2)(x^2-2)$. 
The origin is the only singular point of $X_{\C}$ and thus $X$ is
normal. 
Moreover the blowing-up of the origin gives a resolution of the
singularities of $X$, with exceptional divisor a smooth irreducible
curve with two connected components. 
It is in particular $\Co$-irreducible, therefore it follows again
from Proposition \ref{intclossing} that $\SR(X)$, which is also equal to $\SRR(X)$ since $X$ has only an isolated singularity, is integrally
  closed.
\end{enumerate}
\end{ex}

\begin{rem} Note that even for normal central surfaces, the ring
  $\SRR(X)=\SR(X)$ is not necessarily integrally closed. Consider
  for example the surface $X$ given by $z^2=(x^2+y^2)^2+x^6$ in
  $\R^3$. The origin is the only singular point of $X_{\C}$ and thus
  $X$ is normal. The rational function $f=z/(x^2+y^2)$ satisfies
  the integral equation $f^2=1+x^6/(x^2+y^2)$ with
  coefficients in $\SRR(X)$. 
  As a consequence $f^2$ converges to 1 at the origin, but $f$ has different signs depending on the sign of $z$. Therefore $f$ can not be continuous at the origin.
\end{rem}

\subsection{The case of curves}

Let $X$ be an irreducible
algebraic curve. Recall that in that situation, hereditarily and continuous rational functions (on the central locus) coincide. Let $\pi':X'\rightarrow X$ denote the normalization
map. Note that a normal curve is non-singular and thus automatically central.

In this final section, we aim to determine the integral closure
$\SR(\Cent X)'$ of $\SR(\Cent X)$ in $\K(X)$.
By \cite[Prop. 2.4]{FMQ}, we know that
$\SR(X')$ coincides with $\SO(X')$, so that $\SR(X')$ is an integrally closed ring
by Corollary \ref{intcloslisse} (or Proposition \ref{intclosreel2}). Moreover, the
composition with $\pi'$ induces an inclusion
$\SR(\Cent X)\subset\SR(X')$, so that we obtain a sequence of
inclusions $$\SO(X)\subset \SR(\Cent X)\subset \SO(X').$$
Taking integral closures, we obtain another sequence of inclusions 
\begin{equation}
\label{equ1}
\SO(X)'\subset \SR(\Cent X)'\subset \SO(X'),
\end{equation}
namely $\SR(\Cent X)'$ is an intermediate ring between
$\SO(X)'$ and $\SO(X')$.

Note that, in the particular case when $X$ has totally real normalization, this sequence of inclusions becomes equalities $\SO(X)'= \SR(\Cent X)'= \SO(X')$ by Proposition \ref{egalite}.

\begin{ex}
\label{grospoint}
{\rm Consider the central curve $X=\Z(y^2-x(x^2+y^2))$.
Let $X'$ be its normalization. The curve $X$ has a
  unique singular point obtained by putting together two complex conjugated
  points of $X'_{\C}$ and a point of $X'$. Since the normalization map
  $\pi':X'\rightarrow X$ is a bijection then the composition by $\pi'$
  gives an isomorphism between $\SR(X)$ and $\SR(X')=\SO(X')$ and thus
  $\SR(X)=\SR(X)'=\SO(X')$. Indeed a regular function on $X'$ is constant on the fibers of $\pi'$ and thus induces a (rational) continuous function on $X$.
  Since the fiber of
  $\pi_{\C}:X'_{\C}\rightarrow X_{\C}$ over the singular point of $X$
  is not totally real then $\SR(X)'=\SO(X')\not=
  \SO(X)'$ (Proposition \ref{egalite}). In this example, the first inclusion in
  (\ref{equ1}) is strict and the second one is an equality.}
\end{ex}

We finally relate the integral closure of $\SR(\Cent X)$ with the ring of regular functions on the normalization of $X$.
\begin{thm}
\label{intcloscurve4}
Let
$\pi':X'\rightarrow X$ be the normalization map of an irreducible curve $X$.
Then the integral closure $\SR(\Cent X)'$ of $\SR(\Cent X)$ in $\K(X)$ coincides with the ring $\SO(X')$ of regular functions on the normalization $X'$ of $X$.
\end{thm}

\begin{proof}
Note first that we only have to consider the local case.
It is not difficult to see that the maximal ideals of $\SR(\Cent X)$ are of the form $\m_x=\{f\in\SR(\Cent X)|\;f(x)=0\}$ for $x\in\Cent X$. Indeed, if $X\subset \R^n$ and since $X$ is a curve then we can extend a continuous rational 
function on $\Cent X$ as a continuous rational function on $\R^n$ and apply \cite[Prop. 5.11]{FMQ}.
Let $x\in \Cent X$. 
We consider the fiber
$(\pi')_{\C}^{-1}(x)=\{y_1,\ldots,y_r,z_1,\overline{z_1},\ldots,
z_t,\overline{z_t}\}$ where $r,t$ are integers, the complex
involution is denoted by a bar, $y_1,\ldots,y_r$ correspond to points of $X'$ and $z_1,\overline{z_1},\ldots,
z_t,\overline{z_t}$ is a set of distinct two-by-two points of $X'_{\C}$.
By Proposition \ref{centralsurj2}, we know that $r\geq 1$.
By Proposition \ref{centralsurj2}, it follows that $\pi':X'\to\Cent X$ is a quotient map for the Euclidean topology and thus 
the continuous rational functions on $\Cent X$
correspond to the regular functions on $X'$ which are constant on the fibers
of $\pi':X'\to\Cent X$. It follows
that $$\SR(\Cent X)_{\m_x}\supset\R+\m_{y_1}\cap\cdots\cap \m_{y_r}$$ with $\m_{y_i}=\{f\in\SO(X')_{\m_x}|\;f(y_i)=0\}$.
We want to prove that $$(\SR(\Cent X)_{\m_x})'=(\SR(\Cent X)')_{\m_x}=\SO(X')_{\m_x}.$$
Let $f\in \SO(X')_{\m_x}$. For $i=1,\ldots,r$, there exists
$\alpha_i\in\R$ such that $f-\alpha_i\in
\m_{y_i}$. Consequently the product $\prod_{i=1}^{r}(f-\alpha_i)$ belong to $
\SR(\Cent X)_{\m_x}$, therefore $f$ satisfies an integral equation with coefficient in $\SR(\Cent X)_{\m_x}$ as required.

\end{proof}

\end{document}